\documentclass{article}[utf8, 10pt]%
\usepackage{amsmath}
\usepackage{amsfonts}
\usepackage{amssymb}
\usepackage{mathrsfs}
\usepackage{graphicx}
\usepackage{hyperref}
\usepackage{stmaryrd}
\usepackage{authblk}
\usepackage{enumerate}
\usepackage{bbm}
\usepackage{color}%
\usepackage{esint}
\usepackage{mathtools}
\usepackage[title,titletoc]{appendix}
\usepackage{xpatch,amsthm}
\makeatletter
\xpatchcmd{\@thm}{\fontseries\mddefault\upshape}{}{}{} 
\makeatother
\makeatletter
\newcommand{\subjclass}[2][2020]{%
	\let\@oldtitle\@title%
	\gdef\@title{\@oldtitle\footnotetext{#1 \emph{Mathematics subject classification.} #2}}%
}
\newcommand{\keywords}[1]{%
	\let\@@oldtitle\@title%
	\gdef\@title{\@@oldtitle\footnotetext{\emph{Key words and phrases.} #1.}}%
}
\makeatother
\providecommand{\U}[1]{\protect\rule{.1in}{.1in}}
\newtheorem{theorem}{Theorem}

\newtheorem{definition}[theorem]{Definition}

\newtheorem{lemma}[theorem]{Lemma}

\newtheorem*{question*}{Question}
\newtheorem{proposition}[theorem]{Proposition}
\newtheorem{remark}[theorem]{Remark}

\begin{document}
	\author[a]{Dharmendra Kumar\thanks{dharmendrakumar@uohyd.ac.in}}
	\author[b]{Swarnendu Sil\thanks{swarnendusil@iisc.ac.in}} 
	\affil[a]{School of Mathematics and Statistics\\ University of Hyderabad\\ Hyderabad, India.}
	\affil[b]{Department of Mathematics\\ Indian Institute of Science\\ Bangalore, India. }
	
	\title{BMO estimates for Hodge-Maxwell systems with discontinuous anisotropic coefficients}
	\maketitle 
	
	\begin{abstract}
		We prove up to the boundary $\mathrm{BMO}$ estimates for linear Maxwell-Hodge type systems for $\mathbb{R}^{N}$-valued differential $k$-forms $u$ in $n$ dimensions 
		\begin{align*}
			\left\lbrace \begin{aligned}
				d^\ast \left(  A(x)  du \right)   &= f   &&\text{ in } \Omega,\\
				d^\ast \left( B(x) u\right) &= g &&\text{ in } \Omega, 
			\end{aligned} 
			\right. 
		\end{align*}
		with $ \nu\wedge u$ prescribed on $\partial\Omega,$ where the coefficient tensors $A,B$ are only required to be bounded measurable and in a class of `small multipliers of BMO'. This class neither contains nor is contained in $C^{0}.$ Since the coefficients are allowed to be discontinuous, the usual Korn's freezing trick can not be applied. As an application, we show BMO estimates hold for the time-harmonic Maxwell system in dimension three for a class of discontinuous anisotropic permeability and permittivity tensors. The regularity assumption on the coefficient is essentially sharp. 
	\end{abstract}
	
	\section{Introduction}
	As far as systems of PDEs for a differential $k$-form $u$ on a bounded domain $\Omega \subset \mathbb{R}^{n}$ are concerned, one of the most important first order systems, if not the most, are the so-called Hodge systems 
	\begin{align*}
		\left\lbrace \begin{aligned}
			d u &= f &&\text{ in } \Omega, \\
			d^{\ast} u &=g &&\text{ in } \Omega, \\
			\nu \wedge u &= \nu \wedge u_{0} &&\text{ on } \partial\Omega, 
		\end{aligned}\right. \qquad\text{ and } \qquad  \left\lbrace \begin{aligned}
			d u &= f &&\text{ in } \Omega, \\
			d^{\ast} u &=g &&\text{ in } \Omega, \\
			\nu \lrcorner u &= \nu \lrcorner u_{0} &&\text{ on } \partial\Omega. 
		\end{aligned}\right. 
	\end{align*}
	These systems, which are dual to each other by Hodge duality, are classical and have been studied extensively in a variety of contexts, e.g. Poincar\'{e} lemma, Cauchy-Riemann operators, Dirac operators, Gaffney-Friedrich's inequalities, div-curl lemmas etc,  just to name a few. Elliptic regularity results for these systems can be derived from the Hodge-Morrey decomposition, which itself a consequence of the regularity results for the second order Hodge Laplacian systems ( see \cite{CsatoDacKneuss}, \cite{SchwarzHodge}, \cite{Morrey1966} etc ). Similar remarks are valid for the Hodge-Maxwell systems   
	\begin{align*}
		\left\lbrace \begin{aligned}
			d^\ast du &= f &&\text{ in } \Omega, \\
			d^{\ast} u &=g &&\text{ in } \Omega, \\
			\nu \wedge u &= \nu \wedge u_{0} &&\text{ on } \partial\Omega, 
		\end{aligned}\right. \qquad\text{ and } \qquad  \left\lbrace \begin{aligned}
			d d^{\ast}u &= f &&\text{ in } \Omega, \\
			d u &=g &&\text{ in } \Omega, \\
			\nu \lrcorner u &= \nu \lrcorner u_{0} &&\text{ on } \partial\Omega. 
		\end{aligned}\right. 
	\end{align*}
	However, their natural `rotated' analogue,
	\begin{align}\label{General Maxwell system}
		\left\lbrace \begin{aligned}
			d^\ast \left( A \left(x\right)  du \right) &= f &&\text{ in } \Omega, \\
			d^{\ast} \left( B \left(x\right)  u \right) &=g &&\text{ in } \Omega, \\
			\nu \wedge u &= \nu \wedge u_{0} &&\text{ on } \partial\Omega, 
		\end{aligned}\right. 
	\end{align}
	where $A, B$ are bounded measurable, uniformly elliptic matrices, have received hardly any attention. This is rather surprising, as this is really the system that is relevant for applications. Indeed, one can already see this for the case of time-harmonic Maxwell system itself. The time-harmonic Maxwell's system in bounded domain in $\mathbb{R}^{3}$ can be written as a second order system in $E$ as follows 
	\begin{align*}
		\left\lbrace \begin{aligned}
			\operatorname*{curl} ( \mu^{-1} \operatorname*{curl} E  ) &=  \omega^2 \varepsilon E -i\omega J_{e} + \operatorname*{curl}\left( \mu^{-1} J_{m}\right)    
			&&\text{ in } \Omega, \\
			\operatorname*{div} ( \varepsilon E ) &= \frac{i}{\omega}\operatorname*{div} J_{e} &&\text{ in } \Omega, \\
			\nu \times E &= \nu \times E_{0} &&\text{  on } \partial\Omega,
		\end{aligned} 
		\right. 
	\end{align*}
	which has the same form as \eqref{General Maxwell system}, with $A = \mu^{-1}$ and $B=\varepsilon.$
	
	Regularity results for the system \eqref{General Maxwell system} for a vector-valued differential $k$-form, along with a variety of  related linear systems, are studied systematically in \cite{Sil_linearregularity}, where the matrices $A,B$ are assumed to be sufficiently regular. The techniques employed there is perturbative in nature, i.e. regularity results are derived by using Korn's freezing trick to freeze the coefficients at a point and then deriving the regularity estimates for the case when $A, B$ are constant matrices.   
	
	In the present article, we are interested in deriving $\mathrm{BMO}$ estimates for the system 
	\begin{align}\label{General Maxwell system BMO}
		\left\lbrace \begin{aligned}
			d^\ast \left( A \left(x\right)  du \right) &= d^{\ast}f &&\text{ in } \Omega, \\
			d^{\ast} \left( B \left(x\right)  u \right) &=d^{\ast}g &&\text{ in } \Omega, \\
			\nu \wedge u &= \nu \wedge u_{0} &&\text{ on } \partial\Omega, 
		\end{aligned}\right. 
	\end{align}
	when $A, B$ are allowed to be discontinuous, so that the techniques of \cite{Sil_linearregularity} can no longer be applied. For $0$-forms, the system \eqref{General Maxwell system BMO} reduces to 
	\begin{align*}
		\left\lbrace \begin{aligned}
			\operatorname{div}\left(A\left(x\right)\nabla u\right) &=\operatorname{div} f &&\text{ in } \Omega, \\
			u &=u_{0} &&\text{ on } \partial\Omega.
		\end{aligned}\right. 
	\end{align*} 
	Acquistapace in \cite{Acquistapace_linear_elliptic_discontinuous} derived $\mathrm{BMO}$ estimates for $Du$ assuming $f$ in $\mathrm{BMO}$ when the coefficients of $A$ are in the class $L^{\infty}\cap \mathrm{V}\mathscr{L}_{\left( 1+ \left\lvert \log r \right\rvert \right)^{-1}}^{2,n},$ where $\mathrm{V}\mathscr{L}_{\left( 1+ \left\lvert \log r \right\rvert \right)^{-1}}^{2,n}$ is a space of Campanato-Spanne--Sarason-Janson type, consisting of functions which are `small multipliers of $\mathrm{BMO}$'. This space neither contains nor is contained in $C^{0}.$ See Section \ref{Mean Oscilaltion spaces} and \cite{Spanne_BMO}, \cite{Sarason_VMO}, \cite{Janson_BMO} for more on this. The hypothesis on the coefficient is essentially sharp, as Acquistapace discusses at length in \cite{Acquistapace_linear_elliptic_discontinuous}.  
	
	The main results of the present article are Theorem \ref{Main theorem BMO linear A u0}, Theorem \ref{Main theorem BMO linear AB} and Theorem \ref{Main theorem Maxwell BMO linear AB}. Theorem \ref{Main theorem BMO linear A u0}  should be viewed as a generalization of Acquistapace's results to system \eqref{General Maxwell system BMO} for vector-valued $k$ forms in the case when $B$ is assumed to be the identity. Theorem \ref{Main theorem BMO linear AB} states that when $A, B$ are uniformly Legendre-elliptic and coefficients of $A, B$ are in the class $L^{\infty}\cap \mathrm{V}\mathscr{L}_{\left( 1+ \left\lvert \log r \right\rvert \right)^{-1}}^{2,n},$ then $du$ and $u$ are in $\mathrm{BMO}$ as soon as $f,g, u_{0}, du_{0}$ are $\mathrm{BMO}.$ We derive the regularity result for systems with a linear zeroth order term in Theorem \ref{Main theorem Maxwell BMO linear AB} as a consequence of our estimates. As an application, we have the following theorem for time-harmonic Maxwell system. 
	\begin{theorem}\label{BMO for Maxwell in 3dim}
		Let $\Omega \subset \mathbb{R}^{3}$ be an open, bounded subset with $C^2$ boundary.   Let $E_{0} \in L^{2}\left(\Omega; \mathbb{R}^{3}\right)$ be such that $\operatorname{curl} E_{0}  \in L^{2}\left(\Omega; \mathbb{R}^{3}\right).$ Let $J_{e}, J_{m} \in L^{2}\left(\Omega; \mathbb{R}^{3}\right)$  and suppose  $$J_{e}, J_{m},  E_{0}, \operatorname{curl} E_{0} \in \mathrm{BMO} \left(\Omega; \mathbb{R}^{3}\right).$$ Let $\varepsilon, \mu \in L^{\infty}\cap \mathrm{V}\mathscr{L}_{\left( 1+ \left\lvert \log r \right\rvert \right)^{-1}}^{2,n}\left(\Omega; \mathbb{R}^{3\times 3}\right)$ be uniformly Legendre-elliptic with ellipticity constants $\gamma_{1}, \gamma_{2} >0,$ respectively.
		Let $E, H \in L^{2}\left(\Omega; \mathbb{R}^{3}\right)$ be a weak solution to 
		\begin{align}\label{Maxwell in 3dim}
			\left\lbrace \begin{aligned}
				\operatorname*{curl}  H  &=  i\omega \varepsilon E + J_{e}    
				&&\text{ in } \Omega, \\
				\operatorname*{curl} E &= -i\omega \mu H + J_{m}     &&\text{ in } \Omega, \\
				\nu \times E &= \nu \times E_{0} &&\text{  on } \partial\Omega.
			\end{aligned} 
			\right.
		\end{align}
		Then $E, H \in \mathrm{BMO} \left(\Omega; \mathbb{R}^{3}\right)$ and there exists a constant $ C >0,$ depending only on $\gamma_{1},\gamma_{2}, \omega, \Omega$ and corresponding norms and moduli of $\varepsilon, \mu $ such that we have 
		\begin{align*}
			&\left[\left(E,H\right)\right]_{\mathrm{BMO} \left(\Omega\right)} \\ &\leq C \left( \left\lVert \left(E,H, J_{e}, J_{m}, E_{0}, \operatorname{curl} E_{0} \right) \right\rVert_{L^{2}\left(\Omega\right)} + \left\lVert \left(J_{e}, J_{m}, E_{0}, \operatorname{curl} E_{0}\right)\right\rVert_{\mathrm{BMO} \left(\Omega\right)} \right). 
		\end{align*}
	\end{theorem}
	The hypotheses of the theorem is satisfied when $\varepsilon, \mu$ are Dini-continuous, so the result holds for such coefficient tensors as well. As far as we are aware, unless the tensors $\varepsilon, \mu$ are assumed to be isotropic, the $\mathrm{BMO}$ estimate for this system was known only for H\"{o}lder continuous $\varepsilon, \mu,$ which can be easily deduced from \cite{Sil_linearregularity}. To the best of our knowledge, our result is new even for anisotropic Dini-continuous coefficients.

	The crux of the difficulty to adapt Acquistapace's technique to our setup is twofold. The second order part of our system is $d^{\ast}\left( A\left(x\right)du\right).$ As adding closed form to a solution $u$ would yield another solution unless $u$ is a $0$-form, this operator has an infinite dimensional kernel and thus is neither elliptic nor Fredholm.  To extract the most information about the regularity of solutions locally, it is necessary to use the `gauge invariance' in a clever way. The novelty of the present contribution lies primarily in this judicious exploitation of `gauge freedom'. At the technical level, this is realized via switching weak formulations as needed to have corresponding Poincar\'{e} inequalities to use at different stages of the argument. The other technical point is that though $d$ commutes with the pullback via diffeomorphisms, $d^{\ast}$ does not. Thus, the usual `flattening the boundary' step involves an additional `second order' term $\operatorname{div} \left( S \nabla u \right).$ In \cite{Sil_linearregularity}, a similar term which appeared was tackled using the regularity of the coefficients. Here we needed to achieve the same without using any regularity for $A.$ 
	
	The rest of the article is organized as follows. We detail our notations in Section \ref{notations}. Section \ref{Function spaces} describe the function spaces used and collect some facts about these spaces. Section \ref{Boundary BMO linear} proves our main up to the boundary estimates. Section \ref{main results as corollaries} is concerned with proving our main results using these estimates. 
	\section{Notations}\label{notations}
	We now fix the notations, for further details we refer to
	\cite{CsatoDacKneuss}. Let $n \geq 2,$ $N \geq 1$ and $0 \leq k \leq n$ be integers. \begin{itemize}
		\item We write $\Lambda^{k}\mathbb{R}^{n}$ to denote the vector space of all alternating $k-$linear maps
		$f:\underbrace{\mathbb{R}^{n}\times\cdots\times\mathbb{R}^{n}}_{k-\text{times}%
		}\rightarrow\mathbb{R}.$ For $k=0,$ we set $\Lambda^{0} \mathbb{R}%
		^{n}  =\mathbb{R}.$ Note that $\Lambda^{k}  \mathbb{R}%
		^{n}  =\{0\}$ for $k>n$ and, for $k\leq n,$ $\operatorname{dim}\left(
		\Lambda^{k}  \mathbb{R}^{n}  \right)  ={\binom{{n}}{{k}}}.$ 
		\item For any two finite dimensional vector spaces $X, Y,$ we use the notation $\mathcal{L}\left(X, Y \right)$ to denote the vector space of all linear maps from $X$ to $Y.$
		
		\item We would be dealing with vector-valued forms a lot, we introduce some shorthand notation to avoid clutter. The integers $n \geq 2,$ $N \geq 1$ would remain fixed but arbitrary for the rest.  The only relevant point is the degree of the form.  To this end, for any integer  $0 \leq k \leq n-1,$ we denote 
		\begin{align*}
			\varLambda^{k}:= \Lambda^{k}\mathbb{R}^{n}\otimes \mathbb{R}^{N}.
		\end{align*}

		\item $\wedge,$ $\lrcorner\,,$ $\left\langle \ ;\ \right\rangle $ and,
		respectively, $\ast$ denote the exterior product, the interior product, the
		scalar product and, respectively, the Hodge star operator, extended componentwise in the obvious fashion to vector-valued forms. 
		
		\item Let $\left\{  e_{1},\cdots,e_{n}\right\}  $ be the standard basis of $\mathbb{R}%
		^{n}.$ The dual basis $\left\{  e^{1},\cdots,e^{n}\right\} $ is a basis for $\Lambda^{1}\mathbb{R}^{n}$ and 
		$\left\{  e^{i_{1}}\wedge\cdots\wedge e^{i_{k}}:1\leq i_{1}<\cdots<i_{k}\leq
		n\right\}$	is a basis of $\Lambda^{k}\mathbb{R}^{n}.$ An element $\xi\in\varLambda^{k}$ will therefore be written as%
		\begin{align*}
			\xi=\sum\limits_{j=1}^{N}\sum_{I\in\mathcal{T}^{k}}\xi
			_{I,j}\,e^{I}\otimes e_{j}=\sum\sum\left( e^{i_{1}}\wedge\cdots\wedge e^{i_{k}}\right)\otimes e_{j} %
		\end{align*}
		where $	\mathcal{T}^{k}=\left\{  I=\left(  i_{1}\,,\cdots,i_{k}\right)
		\in\mathbb{N}^{k}:1\leq i_{1}<\cdots<i_{k}\leq n\right\}.$
		
		\item 	Let $\Omega\subset\mathbb{R}^n$ be open, bounded and with $C^{1}$ boundary. $\nu$ will always denote the outward unit normal field to $\partial\Omega,$ which would be identified with abuse of notation, with the $1$-form with same components. 
		An $\mathbb{R}^{N}$-valued differential $k$-form $\omega$ on $\Omega$ is a measurable function $\omega:\Omega\rightarrow\varLambda^{k}.$ The usual Lebesgue, Sobolev and H\"{o}lder spaces are defined componentwise in the usual way and are denoted by their usual symbols. For any measurable subset $A \subset \mathbb{R}^{n}$ with $\left\lvert A \right\rvert < \infty,$  we use the notation $\left( \cdot \right)_{A}$ to denote the integral average over the set $A$, i.e.  
		$$ \left( f \right)_{A} := \frac{1}{\left\lvert A \right\rvert } \int_{A} f\left(x\right)\ \mathrm{d}x := \fint_{A} f\left(x\right)\ \mathrm{d}x \qquad \text{ for any } f \in L^{1}\left( A \right).$$ This notation is also extended componentwise to vector-valued functions. 
		
		\item Two important differential operators on differential forms are
		\begin{definition}
			A $\mathbb{R}^{N}$-valued differential
			$(k+1)$-form $\varphi\in L^{1}_{\text{loc}}(\Omega;\varLambda^{k+1})$
			is called the exterior derivative of $\omega\in
			L^{1}_{\text{loc}}\left(\Omega;\varLambda^{k}\right),$ denoted by $d\omega$,  if
			\begin{align*}
				\int_{\Omega} \eta\wedge\varphi=(-1)^{n-k}\int_{\Omega} d\eta\wedge\omega, \qquad \text{ for all } \eta\in C^{\infty}_{c}\left(\Omega;\varLambda^{n-k-1}\right).
			\end{align*}
			The Hodge codifferential of $\omega\in L^{1}_{\text{loc}}\left(\Omega;\varLambda^{k}\right)$ is
			an $\mathbb{R}^{N}$-valued $(k-1)$-form, denoted $d^{\ast}\omega\in L^{1}_{\text{loc}}\left(\Omega;\varLambda^{k-1}\right),$
			defined as
			$$
			d^{\ast}\omega:=(-1)^{nk+1} \ast d \ast \omega. $$
		\end{definition}
		See \cite{CsatoDacKneuss} for the properties of  these operators.

		\item We shall use the following two ellipticity conditions for matrix fields. 
		\begin{definition}\label{legendre-hadamard condition}
			A linear map $A:\varLambda^{k}\otimes \mathbb{R}^{n} \rightarrow \varLambda^{k}\otimes \mathbb{R}^{n}$ is said to satisfy the \textbf{Legendre-Hadamard condition} if there exists a constant $ \gamma >0$ such that 
			$$ \langle A ( a\otimes b ) \  ;\  a\otimes b \rangle  \geq \gamma \left\vert  a \right\vert^{2}\left\lvert b \right\vert^{2}, 
			\qquad  \text{ for every } a \in \mathbb{R}^{n}, b \in \varLambda^{k} . $$
		\end{definition}
		\begin{definition}\label{legendre condition}
			A bounded measurable map $A \in L^{\infty}\left( \Omega; \mathcal{L}(\varLambda^{k},\varLambda^{k}) \right)$ is called \textbf{uniformly Legendre elliptic} if there exists a constant $ \gamma >0$ such that  we have 
			$$ \gamma \left\vert  \xi \right\vert^{2} \leq \langle A (x)  \xi \  ;\  \xi \rangle  \leq \left\lVert A \right\rVert_{L^{\infty}}\left\vert  \xi \right\vert^{2} 
			\quad  \text{ for every } \xi \in \varLambda^{k} \text{ and for   a.e. }  x \in \Omega.$$ 
		\end{definition}
		
		\item For any $x \in \mathbb{R}^{n},$ any $\rho >0,$ $B_{\rho}\left(x\right)$ denotes the open ball of radius $\rho >0$ around $x.$ 
		If $x \in \partial\mathbb{R}^{n}_{+},$ $B_{\rho}^{+}\left(x\right)$ will denote the half-ball centered around $x$ in the upper half space, i.e 
		$$B_{\rho}^{+}\left(x\right) = \lbrace  y \in \mathbb{R}^{n}: y_{n} > 0, \lvert y -x \rvert < \rho \rbrace .$$
		Let $\Gamma_{\rho}\left(x\right)$  and $C_{\rho}\left(x\right)$ denote the flat part and the curved part, respectively, of the boundary of the half ball $B_{\rho}^{+}\left(x\right).$ Also, for any open set $\Omega \subset \mathbb{R}^{n},$ we denote the set 
		\begin{align*}
			\Omega_{\rho}\left(x\right) := \Omega \cap B_{\rho}\left(x\right) \quad \text{ and } \quad 	\Omega^{+}_{\rho}\left(x\right) := \Omega \cap B_{\rho}^{+}\left(x\right) (\text{when } x \in \partial\mathbb{R}^{n}_{+}). 
		\end{align*}
		We suppress writing the center when $x=0 \in \mathbb{R}^{n}.$
		\item We reserve the notation $\theta$ to denote the function $\theta: \mathbb{R}_{+} \rightarrow \mathbb{R}_{+},$ defined as 
		\begin{align*}
			\theta \left(r \right):= \sup\limits_{0 < \rho \leq r} \rho \left( 1 + \left\lvert \log \rho \right\rvert \right).
		\end{align*}
		
	\end{itemize}

	\section{Function spaces}\label{Function spaces}
	\subsection{Gaffney and Poincar\'{e} type inequalities}
	Let $\Omega \subset \mathbb{R}^{n}$ open, bounded,$ C^{2}.$ The spaces $W^{d,2}_{T}\left(  \Omega;\varLambda^{k}\right)$ and  $W_{T}^{1,2}\left(  \Omega;\varLambda^{k}\right)  $ are defined as ( see \cite{CsatoDacKneuss} )%
	\begin{align*}
		W^{d,2}_{T}\left(  \Omega;\varLambda^{k}\right) &= \left\lbrace \omega \in L^{2}_{T}\left(  \Omega;\varLambda^{k}\right): du \in L^{2}_{T}\left(  \Omega;\varLambda^{k+1}\right) \text{ and } \nu\wedge\omega=0\text{ on
		}\partial\Omega \right\rbrace, \\ 
		W_{T}^{1,2}\left(  \Omega;\varLambda^{k}\right)  &=\left\{  \omega\in
		W^{1,2}\left(  \Omega;\varLambda^{k}\right)  :\nu\wedge\omega=0\text{ on
		}\partial\Omega\right\}. 
	\end{align*}
	The  subspaces $W_{d^{\ast}, T}^{1,2}(\Omega; \varLambda^{k})$ and $\mathcal{H}^{k}_{T}\left(  \Omega;\varLambda^{k}\right)$ are defined as  
	\begin{align*}
		W_{d^{\ast}, T}^{1,2}(\Omega; \varLambda^{k}) &= \left\lbrace \omega \in W_{T}^{1,2}(\Omega; \varLambda^{k}) : d^{\ast}\omega = 0 \text{ in }
		\Omega \right\rbrace \\
		\mathcal{H}^{k}_{T}\left(  \Omega;\varLambda^{k}\right)  &=\left\{  \omega\in
		W_{T}^{1,2}\left(  \Omega;\varLambda^{k}\right)  :d\omega=0\text{ and }%
		d^{\ast}\omega=0\text{ in }\Omega\right\}
	\end{align*}
	For half-balls, we need the following subspace. 
	\begin{align*}
		W&_{T, \text{flat}}^{1,2}(B_{R}^{+}\left(x_{0}\right) ; \varLambda^{k}) \\&= \left\lbrace \psi \in W^{1,2}(B_{R}^{+}\left(x_{0}\right) ; \varLambda^{k}):
		e_n \wedge \psi = 0 \text{ on }  \Gamma_{R}\left(x_{0}\right), \psi = 0 \text{ on } C_{R}\left(x_{0}\right) \right\rbrace . 
	\end{align*}
	The following Gaffney inequality follows from the standard Gaffney inequality by a contradiction argument ( see Step 1 of the proof of Theorem 6.7 in \cite{CsatoDacKneuss} ). 
	\begin{proposition}[Gaffney inequality]\label{Gaffney with no L2 term}
		Let $\Omega \subset \mathbb{R}^{n}$ be open, bounded, $C^{2}.$ There exists a constant $C = C\left(\Omega, n, N, k\right) >0$ such that 
		\begin{align*}
			\left\lVert \nabla u \right\rVert^{2}_{L^{2}\left(\Omega\right)} \leq C \left( 	\left\lVert du \right\rVert^{2}_{L^{2}\left(\Omega\right)} + 	\left\lVert d^{\ast} u \right\rVert^{2}_{L^{2}\left(\Omega\right)}\right) \quad \text{ for all }  u \in W^{1,2}_{T} \cap \left( \mathcal{H}^{k}_{T}\right)^{\perp}.   
		\end{align*}  
	\end{proposition}
	We shall also need  following Poincar\'{e} type inequality. 
	\begin{proposition}[Poincar\'{e} inequality]\label{poincare gaffney}
		Let $\Omega \subset \mathbb{R}^{n}$ be open, bounded, $C^{2}.$ There exists a constant $C = C\left(\Omega, n, N, k\right) >0$ such that 
		\begin{align}\label{poincareineq}
			\left\lVert  u \right\rVert^{2}_{W^{1,2}\left(\Omega; \varLambda^{k}\right)} \leq C  	\left\lVert du \right\rVert^{2}_{L^{2}\left(\Omega; \varLambda^{k+1}\right)} \qquad \text{ for all } u \in W^{1,2}_{d^{\ast}, T} \cap \left( \mathcal{H}_{T}^{k}\right)^{\perp}.
		\end{align}
	\end{proposition}
	\begin{proof}
		In view of Proposition \ref{Gaffney with no L2 term}, it is enough to show the Poincar\'{e} inequality
		$$ \lVert u \rVert^{2}_{L^{2}} \leq c \lVert \nabla u \rVert^{2}_{L^{2}} \quad \text{ for all }u \in W^{1,2}_{d^{\ast}, T}\left( \Omega ; \varLambda^{k}\right) \cap \left( \mathcal{H}^{k}_{T}\left(  \Omega;\varLambda^{k}\right)\right)^{\perp}.$$ 
		But this follows by a contradiction argument as in Proposition 1 in \cite{Sil_nonlinearStein}. 
	\end{proof}
	\subsection{Weak formulations and gauge fixing}\label{Gauge fixing section}
	Our results depend crucially on our ability to switch weak formulations. Note that 
	$C_{c}^{\infty}\left(\Omega\right)$ is not dense 
	in  $W^{1,2}_{d^{\ast},T}\left(\Omega\right)$ and hence it is far from obvious that weak formulations $W^{1,2}_{d^{\ast},T}$ can be localized.   
	
	\begin{proposition}\label{weak formulation change linear lemma}
		Let $\Omega \subset \mathbb{R}^{n}$ be open, bounded and $C^{2}.$ Let $F \in L^{2}\left( \Omega; \varLambda^{k+1}\right)$ and $A \in L^{\infty}\left( \Omega; \mathcal{L}\left(\varLambda^{k+1}; \varLambda^{k+1} \right)\right).$ The following are all equivalent. 
		\begin{enumerate}[(a)]
			\item $u \in  W^{1,2}\left( \Omega; \varLambda^{k}\right).$ satisfies 
			\begin{align*}
				\int_{\Omega}\langle A(x)du;d\phi\rangle =\int_{\Omega}\langle F;d\phi\rangle \quad \text{ for all } \phi\in W^{1,2}_{d^{\ast}, T}\left(\Omega;\varLambda^k\right).
			\end{align*}
			\item $u \in  W^{1,2}\left( \Omega; \varLambda^{k}\right).$ satisfies \begin{align*}
				\int_{\Omega}\langle A(x)du;d\psi\rangle =\int_{\Omega}\langle F;d\psi\rangle \quad \text{ for all } \psi\in W^{1,2}_{0}\left(\Omega;\varLambda^k\right).
			\end{align*}
			\item $u \in  W^{1,2}\left( \Omega; \varLambda^{k}\right).$
			\begin{align*}
				\int_{\Omega}\langle A(x)du;d\psi\rangle =\int_{\Omega}\langle F;d\psi\rangle \quad \text{ for all } \psi\in W^{1,2}_{T}\left(\Omega;\varLambda^k\right).
			\end{align*}
		\end{enumerate}
	\end{proposition}
	\begin{proof}
		\textbf{(a) $\Rightarrow$ (b) } Given $ \psi \in W_{0}^{1,2}\left( \Omega; \varLambda^{k}\right),$ 
		we first find $\alpha \in W^{2,2}\left( \Omega; \varLambda^{k}\right)$ which solves the system
		\begin{align*}
			\left\lbrace \begin{aligned}
				d^\ast d \alpha  &= -d^\ast \psi &&\text{ in } \Omega,\\
				d^\ast \alpha &= 0 &&\text{ in } \Omega, \\
				\nu\wedge \alpha &= 0 &&\text{  on } \partial \Omega,
			\end{aligned} 
			\right. 
		\end{align*}
		Since since $\nu \wedge  \alpha = 0$ on $\partial \Omega$ implies $\nu \wedge  d\alpha = 0$ on $\partial \Omega.$ the result follows by setting $\phi = \psi + d \alpha.$\smallskip 
		
		\textbf{(b) $\Rightarrow$ (c) } Given any $\phi \in W^{1,2}_{T}\left(\Omega; \varLambda^{k}\right),$ using Theorem 8.16 in \cite{CsatoDacKneuss}, we find $\psi \in W_{0}^{1,2}\left( \Omega; \varLambda^{k}\right)$ by solving the system 
		\begin{align*}
			\left\lbrace \begin{aligned}
				d\psi &= d\phi &&\text{ in } \Omega, \\
				\psi &= 0 &&\text{ on } \partial\Omega. 
			\end{aligned}\right. 
		\end{align*}
		
		Since $W^{1,2}_{d^{\ast}, T} \subset W^{1,2}_{T},$ (c) $\Rightarrow$ (a) is trivial and the proof is finished. 	
	\end{proof}
	\begin{remark}
		Proposition \ref{weak formulation change linear lemma} holds because the equation is gauge-invariant, i.e. invariant under translation by kernel of $d$ and the following equality 
		\begin{align*}
			d \left( W^{1,2}_{d^{\ast}, T}\right) = 	d \left( W^{1,2}_{0}\right) = 	d \left( W^{1,2}_{T}\right). 
		\end{align*}
	\end{remark}
	
	\subsection{Existence results}
	We record here an existence result that we would need. 
	\begin{proposition}\label{existence in W012}
		Let $\Omega \subset \mathbb{R}^{n}$ be open, bounded and Lipschitz. Let $A: \varLambda^{k+1} \rightarrow \varLambda^{k+1}$ be Legendre elliptic with constant $\gamma >0.$ For any $f \in L^{2}\left( \Omega ; \varLambda^{k}\right)$ and any $F \in L^{2}\left( \Omega ; \varLambda^{k}\otimes\mathbb{R}^{n}\right),$ there exists unique weak solution to the following 
		\begin{align*}
			\left\lbrace \begin{aligned}
				d^{\ast}\left(A du\right) + d d^{\ast}u &= f + \operatorname{div}F &&\text{ in } \Omega, \\
				u &=0 &&\text{ on } \partial\Omega. 
			\end{aligned}\right. 
		\end{align*}
	\end{proposition}
	\begin{proof}
		Define the linear map $\tilde{A}: \varLambda^{k}\otimes \mathbb{R}^{n} \rightarrow \varLambda^{k}\otimes \mathbb{R}^{n}$ by the pointwise condition
		\begin{align*}
			\left\langle \tilde{A} \left(a_{1} \otimes b_{1}\right); a_{2} \otimes b_{2} \right\rangle := 	\left\langle A \left(a_{1} \wedge b_{1}\right); a_{2} \wedge b_{2} \right\rangle + 	\left\langle a_{1} \lrcorner b_{1}; a_{2} \lrcorner b_{2} \right\rangle
		\end{align*}
		for every $a_{1}, a_{2} \in \mathbb{R}^{n},$ $b_{1}, b_{2} \in \varLambda^{k},$ extended by linearity. Using with the algebraic identity ( see Proposition 2.16 in \cite{CsatoDacKneuss} ), $a \lrcorner \left( a \wedge b \right) + a \wedge \left( a \lrcorner b \right) = \left\lvert a \right\rvert^{2} b, $ it is easy to check that the constant tensor $\tilde{A}$ is Legendre-Hadamard elliptic. Now standard arguments establish existence and uniqueness of $u \in W_{0}^{1,2}\left( \Omega ; \varLambda^{k}\right)$  such that 
		\begin{align*}
			\int_{\Omega}\left\langle \tilde{A}\nabla u, \nabla \phi \right\rangle = \int_{\Omega} \left\langle f, \phi \right\rangle + \int_{\Omega} \left\langle F, \nabla \phi \right\rangle \qquad \text{ for all } \phi \in W_{0}^{1,2}\left( \Omega ; \varLambda^{k}\right).
		\end{align*}
		But this completes the proof as we have 
		\begin{align*}
			\int_{\Omega}\left\langle \tilde{A}\nabla u, \nabla v \right\rangle = \int_{\Omega}\left\langle Adu, dv \right\rangle + \int_{\Omega} \left\langle d^{\ast}u, d^{\ast}v \right\rangle  \qquad \text{ for all } u,v \in W_{0}^{1,2}\left( \Omega ; \varLambda^{k}\right).
		\end{align*} This can be proved by Fourier transform when $u, v \in C_{c}^{\infty}\left( \Omega ; \varLambda^{k}\right)$ and the general case follows by density. 
	\end{proof}
	\subsection{Mean Oscillation spaces}\label{Mean Oscilaltion spaces}
	Let $\Omega \subset \mathbb{R}^{n}$ be open, bounded and Lipschitz. For $1\leq p <  \infty$ and $0 \leq \lambda \leq n+2,$  $\mathscr{L}^{p,\lambda}\left(\Omega\right) $ denotes the Campanato space ( see \cite{Campanato_ellipticsystem_campanatospace} ) of all $f \in L^{p}\left(\Omega\right)$ such that 
	$$ [f ]_{\mathscr{L}^{p,\lambda}\left(\Omega\right)}^{p} := \sup\limits_{\rho >0, x \in \Omega} 
	\rho^{-\lambda} \fint_{\Omega_{\rho}\left(x\right)}\left\lvert f\left(y\right) - \left(f\right)_{\Omega_{\rho}\left(x\right)}\right\rvert^{p}\ \mathrm{d}y < \infty, $$ where 
	endowed with the norm 
	$ \lVert f \rVert_{\mathscr{L}^{p,\lambda}\left(\Omega\right)} := \lVert  f \rVert_{L^{p}(\Omega)} +  
	\left[ f \right]_{\mathscr{L}^{p,\lambda}\left(\Omega\right)}.$
	The space $\mathscr{L}^{1,n}\left(\Omega\right)$ is the space $\mathrm{BMO}\left(\Omega\right)$ and $ \left[ f\right]_{\mathrm{BMO}\left(\Omega\right)} = [f ]_{\mathscr{L}^{1,n}\left(\Omega\right)}.$	
	By John-Nirenberg inequality, $	\mathscr{L}^{p,n}\left(\Omega\right) \simeq \mathrm{BMO}\left(\Omega\right)$  for any $1 \leq p < \infty,$ with equivalent seminorms ( see \cite{giaquinta-martinazzi-regularity} ). We record two estimates about integral averages of $\mathrm{BMO}$ functions. 
	\begin{proposition}[Proposition 1.15, \cite{Acquistapace_linear_elliptic_discontinuous}]\label{integral average by bmo norm interior}
		Let $\Omega \subset \mathbb{R}^{n}$ be open and bounded and let $\Omega^{'} \subset \subset \Omega$ be open. Let $f \in \mathrm{BMO}\left(\Omega\right).$ Then there exists a constant $C = C \left(n\right)>0$ such that for any $ 0 < \sigma_{0}< \min \left\lbrace 2, \operatorname{dist}\left( \Omega^{'}, \partial \Omega \right)/16 \right\rbrace,$ we have 
		\begin{align*}
			\left\lvert \left(f\right)_{B_{r}\left(x\right)}\right\rvert  \leq C \left\lbrace \left(1 + \left\lvert \log r \right\rvert\right)\left[ f\right]_{\mathrm{BMO}\left(\Omega\right)} + \sigma_{0}^{-\frac{n}{2}}\left\lVert f \right\rVert_{L^{2}\left( \Omega \right)}\right\rbrace , 
		\end{align*}
		for all $0 < r< \sigma_{0}$ and for all $x \in \Omega^{'}.$ 
	\end{proposition}
	\begin{proposition}[Proposition 1.16, \cite{Acquistapace_linear_elliptic_discontinuous}]\label{integral average by bmo norm boundary}
		Let $x_{0} \in \partial\mathbb{R}^{n}_{+}$ and $R >0.$ Let $f \in \mathrm{BMO}\left( B_{R}^{+}\left(x_{0}\right)\right).$ Then there exists a constant $C = C \left(n\right)>0$ such that  for any $ 0 < \sigma_{0}< R/8,$ we have  
		\begin{align*}
			\left\lvert \left(f\right)_{B^{+}_{r}\left(x\right)}\right\rvert  \leq C \left\lbrace \left(1 + \left\lvert \log r \right\rvert\right)\left[ f\right]_{\mathrm{BMO}\left(B_{R}^{+}\left(x_{0}\right)\right)} + \sigma_{0}^{-\frac{n}{2}}\left\lVert f \right\rVert_{L^{2}\left( B_{R}^{+}\left(x_{0}\right) \right)}\right\rbrace , 
		\end{align*}
		for all $0 < r< \sigma_{0}$ and for all $x \in \Gamma_{R/2}\left(x_{0}\right).$
	\end{proposition}
	We would crucially use the following generalized Campanato type spaces.  
	\begin{definition}
		We define the space $\mathscr{L}_{\left( 1+ \left\lvert \log r \right\rvert \right)^{-1}}^{2,n}\left(\Omega\right)$ as the vector space of all functions $f \in L^{2}\left(\Omega\right)$ such that 
		\begin{align*}
			\left[ f \right]_{\mathscr{L}_{\left( 1+ \left\lvert \log r \right\rvert \right)^{-1}}^{2,n}\left(\Omega\right)}:= \sup\limits_{  \rho >0, x \in \Omega }   \left( 1+ \left\lvert \log \rho \right\rvert\right) \left[ \fint_{\Omega_{\rho}\left(x\right)}\left\lvert f\left(y\right) - \left(f\right)_{\Omega_{\rho}\left(x\right)}\right\rvert^{2}\ \mathrm{d}y \right]^{\frac{1}{2}} < \infty, 
		\end{align*}
		equipped with the norm 
		\begin{align*}
			\left\lVert f\right\rVert_{\mathscr{L}_{\left( 1+ \left\lvert \log r \right\rvert \right)^{-1}}^{2,n}\left(\Omega\right)} : = \left\lVert f \right\rVert_{L^{2}\left(\Omega\right)} + \left[ f \right]_{\mathscr{L}_{\left( 1+ \left\lvert \log r \right\rvert \right)^{-1}}^{2,n}\left(\Omega\right)}. 
		\end{align*}
		For any $f \in \mathscr{L}_{\left( 1+ \left\lvert \log r \right\rvert \right)^{-1}}^{2,n}\left(\Omega\right),$ the mean oscillation `modulus' of $f$ is defined as the function $\Theta^{f}:\mathbb{R}_{+} \rightarrow \mathbb{R}_{+},$ defined by 
		\begin{align*}
			\Theta^{f} \left( r \right):= \sup\limits_{0 < \rho \leq r, x \in \Omega}   \left( 1+ \left\lvert \log \rho \right\rvert\right) \left[ \fint_{\Omega_{\rho}\left(x\right)}\left\lvert f\left(y\right) - \left(f\right)_{\Omega_{\rho}\left(x\right)}\right\rvert^{2}\ \mathrm{d}y \right]^{\frac{1}{2}}
		\end{align*} 
		We denote by $\mathrm{V}\mathscr{L}_{\left( 1+ \left\lvert \log r \right\rvert \right)^{-1}}^{2,n}\left(\Omega\right)$ the subspace defined by 
		\begin{align*}
			\mathrm{V}\mathscr{L}_{\left( 1+ \left\lvert \log r \right\rvert \right)^{-1}}^{2,n}\left(\Omega\right):= \left\lbrace f \in \mathscr{L}_{\left( 1+ \left\lvert \log r \right\rvert \right)^{-1}}^{2,n}\left(\Omega\right): \lim\limits_{r\rightarrow 0}\Theta^{f}\left(r\right)=0 \right\rbrace.  
		\end{align*}
	\end{definition}
	\begin{remark}\label{Remark about Mean Oscillation space properties}
		Let us now record a few facts about these spaces. 
		\begin{enumerate}[(i)]
			\item $\mathscr{L}_{\left( 1+ \left\lvert \log r \right\rvert \right)^{-1}}^{2,n}$ is Banach and $\mathrm{V}\mathscr{L}_{\left( 1+ \left\lvert \log r \right\rvert \right)^{-1}}^{2,n}$ is a proper subspace and
			\begin{align*}
				\mathrm{V}\mathscr{L}_{\left( 1+ \left\lvert \log r \right\rvert \right)^{-1}}^{2,n}\left(\Omega\right) = \overline{C^{\infty}\left(\overline{\Omega}\right)}^{	\left\lVert \cdot \right\rVert_{\mathscr{L}_{\left( 1+ \left\lvert \log r \right\rvert \right)^{-1}}^{2,n}\left(\Omega\right)}}. 
			\end{align*}
			Moreover, if $f \in L^{\infty}\left(\Omega\right)\cap \mathrm{V}\mathscr{L}_{\left( 1+ \left\lvert \log r \right\rvert \right)^{-1}}^{2,n}\left(\Omega\right),$ then there exists a sequence $\left\lbrace f_{s}\right\rbrace_{s \in \mathbb{N}} \subset C^{\infty}\left(\overline{\Omega}\right)$ such that 
			\begin{align*}
				\left\lVert f_{s}\right\rVert_{L^{\infty}\left(\Omega\right)} &\leq \left\lVert f \right\rVert_{L^{\infty}\left(\Omega\right)} &&\text{ for all } s \in \mathbb{N}, \\
				\Theta^{f_{s}}\left( \sigma \right) &\leq 	c \left( \Theta^{f}\left( \sigma \right) + \left\lVert f \right\rVert_{L^{\infty}\left(\Omega\right)}\sigma^{n} \right) &&\text{ for all } s \in \mathbb{N},
			\end{align*}
			for some constant $c=c\left(n\right)>0$ and 
			\begin{align*}
				\lim\limits_{s \rightarrow \infty} \left( \left\lVert f_{s} - f\right\rVert_{\mathscr{L}_{\left( 1+ \left\lvert \log r \right\rvert \right)^{-1}}^{2,n}\left(\Omega\right)} + \left\lVert f_{s} - f \right\rVert_{L^{p}\left(\Omega\right)}\right) = 0 \quad \text{ for all } 1\leq p < \infty.
			\end{align*}
			See Proposition 1.2 and Remark 1.5 in \cite{Acquistapace_linear_elliptic_discontinuous} for proofs.
			\item  The space of `multipliers' of $\mathrm{BMO}\left( \Omega\right)$ is $L^{\infty}\left(\Omega\right)\cap \mathscr{L}_{\left( 1+ \left\lvert \log r \right\rvert \right)^{-1}}^{2,n}\left(\Omega\right),$ i.e. 
			\begin{align*}
				\mathcal{M}\left( \mathrm{BMO}\left( \Omega\right)\right) \simeq L^{\infty}\left(\Omega\right)\cap \mathscr{L}_{\left( 1+ \left\lvert \log r \right\rvert \right)^{-1}}^{2,n}\left(\Omega\right)
			\end{align*}
			with equivalent norms. See Theorem 2 in \cite{Janson_BMO}. 
			\item The sets $C\left(\overline{\Omega}\right)\setminus \mathscr{L}_{\left( 1+ \left\lvert \log r \right\rvert \right)^{-1}}^{2,n}\left(\Omega\right)$ and $L^{\infty}\left(\Omega\right)\cap \mathrm{V}\mathscr{L}_{\left( 1+ \left\lvert \log r \right\rvert \right)^{-1}}^{2,n}\left(\Omega\right)\setminus C\left(\overline{\Omega}\right)$ are both non-empty. See Proposition 1.9 in \cite{Acquistapace_linear_elliptic_discontinuous}. 
			\item The set of Dini continuous functions in $\overline{\Omega}$ is a proper subset of $C\left(\overline{\Omega}\right)\cap \mathrm{V}\mathscr{L}_{\left( 1+ \left\lvert \log r \right\rvert \right)^{-1}}^{2,n}\left(\Omega\right).$ See Proposition 1.10 in \cite{Acquistapace_linear_elliptic_discontinuous}.
		\end{enumerate}
	\end{remark}
	\begin{proposition}[Proposition 1.13, \cite{Acquistapace_linear_elliptic_discontinuous}]\label{Lphi majorizes Np interior}
		Let $\Omega \subset \mathbb{R}^{n}$ be open and bounded and let $\Omega^{'} \subset \subset \Omega$ be open. Let $f \in \mathscr{L}_{\left( 1+ \left\lvert \log r \right\rvert \right)^{-1}}^{2,n}\left( \Omega\right).$ Then for any $1 \leq p < \infty,$ there exists a constant $C = C \left(n, p\right)>0$ such that for any $ 0 < \sigma_{0}< \operatorname{dist}\left( \Omega^{'}, \partial \Omega \right)/16,$ we have 
		\begin{align*}
			\left(1 + \left\lvert \log r \right\rvert\right)\left[ \fint_{B_{r}\left(x\right)}\left\lvert f - \left(f\right)_{B_{r}\left(x\right)}\right\rvert^{p}\right]^{\frac{1}{p}} \leq C \Theta^{f}\left( \sigma_{0}\right), 
		\end{align*}
		for all $0 < r< \sigma_{0}$ and for all $x \in \Omega^{'}.$ 
	\end{proposition}
	\begin{proposition}[Proposition 1.14, \cite{Acquistapace_linear_elliptic_discontinuous}]\label{Lphi majorizes Np boundary}
		Let $x_{0} \in \partial\mathbb{R}^{n}_{+}$ and $R >0.$ Let $f \in \mathscr{L}_{\left( 1+ \left\lvert \log r \right\rvert \right)^{-1}}^{2,n}\left( B_{R}^{+}\left(x_{0}\right)\right).$ Then for any $1 \leq p < \infty,$ there exists a constant $C = C \left(n, p\right)>0$ such that  for any $ 0 < \sigma_{0}< R/16,$ we have the following two estimates. 
		\begin{enumerate}[(i)]
			\item For all $0 < r< \sigma_{0}$ and for all $x \in B_{R/2}^{+}\left(x_{0}\right)$ such that $B_{r}\left(x\right) \subset \subset B_{R}^{+}\left(x_{0}\right),$ we have 
			\begin{align*}
				\left(1 + \left\lvert \log r \right\rvert\right)\left[ \fint_{B_{r}\left(x\right)}\left\lvert f - \left(f\right)_{B_{r}\left(x\right)}\right\rvert^{p}\right]^{\frac{1}{p}} \leq C \Theta^{f}\left( \sigma_{0}\right). 
			\end{align*}
			
			\item For all $0 < r< \sigma_{0}$ and for all $x \in \Gamma_{R/2}\left(x_{0}\right),$ we have \begin{align*}
				\left(1 + \left\lvert \log r \right\rvert\right)\left[ \fint_{B^{+}_{r}\left(x\right)}\left\lvert f - \left(f\right)_{B^{+}_{r}\left(x\right)}\right\rvert^{p}\right]^{\frac{1}{p}} \leq C \Theta^{f}\left( \sigma_{0}\right).
			\end{align*}
		\end{enumerate}
	\end{proposition}
	The spaces we defined here, along with their claimed properties, extend componentwise to functions that take values in finite dimensional vector spaces. 	
	\begin{lemma}\label{inverse also elliptic and same space}
		Let $A \in L^{\infty} \cap \mathrm{V}\mathscr{L}_{\left( 1+ \left\lvert \log r \right\rvert \right)^{-1}}^{2,n}\left( \Omega; \mathcal{L}\left( \varLambda^{k+1}; \varLambda^{k+1}\right)\right)$ is uniformly Legendre-elliptic. Define $A^{-1}:\Omega \rightarrow \mathcal{L}\left( \varLambda^{k+1}; \varLambda^{k+1}\right)$  by 
		$ A^{-1}\left(x \right) = \left[ A\left(x\right)\right]^{-1}$  for a.e. $x \in \Omega.$ Then 
		$A^{-1} \in L^{\infty} \cap \mathrm{V}\mathscr{L}_{\left( 1+ \left\lvert \log r \right\rvert \right)^{-1}}^{2,n}$ is uniformly Legendre elliptic. 
	\end{lemma}
	\begin{proof}
		By equivalence of norms on finite dimensional spaces, $A^{-1}$ is bounded and uniformly Legendre elliptic. Note the estimate in the operator norms 
		\begin{align*}
			\left\lVert A^{-1}\left(x\right) - \left( A\right)_{\Omega_{\rho}\left(x_{0}\right)}^{-1} \right\rVert_{\mathrm{op}} \leq \left\lVert A^{-1}\left(x\right)\right\rVert_{\mathrm{op}}\left\lVert A\left(x\right) - \left( A\right)_{\Omega_{\rho}\left(x_{0}\right)} \right\rVert_{\mathrm{op}}\left\lVert \left( A\right)_{\Omega_{\rho}\left(x_{0}\right)}^{-1}\right\rVert_{\mathrm{op}}.
		\end{align*}
		From this, by minimality of integral averages and again equivalence of norms, for any $x \in \Omega$ and any $\rho >0,$ we have 
		\begin{align*}
			\fint_{\Omega_{\rho}\left(x\right)} \left\lvert A^{-1} \left(y \right) - \left( A^{-1} \right)_{\Omega_{\rho}\left(x\right)}\right\rvert^{2}\ \mathrm{d} y &\leq 	\fint_{\Omega_{\rho}\left(x\right)} \left\lvert A^{-1} \left(y \right) - \left( A \right)_{\Omega_{\rho}\left(x\right)}^{-1}\right\rvert^{2}\ \mathrm{d} y \\
			&\leq \frac{c}{\gamma^{4}}	\fint_{\Omega_{\rho}\left(x\right)} \left\lvert A \left(y \right) - \left( A \right)_{\Omega_{\rho}\left(x\right)}\right\rvert^{2}\ \mathrm{d} y. 
		\end{align*}
		The claimed result follows. 
	\end{proof}
	The following lemma is easy to establish ( see \cite{Campanato_ellipticsystem_campanatospace} ).  
	\begin{lemma}\label{composition with diffeo and product with determinant}
		Let $\Omega_{1}, \Omega_{2} \subset \mathbb{R}^{n}$ be bounded open subsets and let $\Phi : \overline{\Omega_{1}} \rightarrow \overline{\Omega_{2}}$ be a $C^{2}$ diffeomorphism. Let $A \in L^{\infty} \cap \mathrm{V}\mathscr{L}_{\left( 1+ \left\lvert \log r \right\rvert \right)^{-1}}^{2,n}\left( \Omega_{2}; \mathcal{L}\left( \varLambda^{k+1}; \varLambda^{k+1}\right)\right).$ Define 
		\begin{align*}
			\Tilde{A}\left( y\right)&=\left\lvert \operatorname{det} D\Phi(y) \right\rvert A \left( \Phi(y)\right) &&\text{ for a.e. } y \in \Omega_{1}.
		\end{align*}Then $\Tilde{A} \in L^{\infty} \cap \mathrm{V}\mathscr{L}_{\left( 1+ \left\lvert \log r \right\rvert \right)^{-1}}^{2,n}\left( \Omega_{1}; \mathcal{L}\left( \varLambda^{k+1}; \varLambda^{k+1}\right)\right) $ and there exist  constants $c_{0}, c_{1}>0$, depending only on $\Phi,n, k, N,$ such that  we have the following estimate:
		\begin{align*}
			\Theta^{\tilde{A}}\left(\sigma \right) \leq c_{0}\left( \Theta^{A}\left( c_{1}\sigma \right) + \theta \left(\sigma\right)\left\lVert A \right\rVert_{L^{\infty}\left(\Omega_{2}\right)} \right) &&\text{ for all } \sigma >0. 
		\end{align*}
	\end{lemma}
	
	\section{Crucial estimates}\label{Boundary BMO linear}
	Our goal in this section is to prove the main estimates. 
	\subsection{Boundary estimates}
	\begin{lemma}[Boundary estimate]\label{boundary estimate lemma}
		Let $\partial\Omega$ be $C^{2}$ and let $F \in L^{2}\left(\Omega;\varLambda^{k+1}\right).$ Let 
		$A \in C^{\infty}\left( \overline{\Omega}; \mathcal{L}\left(\varLambda^{k+1}; \varLambda^{k+1} \right)\right)$ be uniformly Legendre elliptic with ellipticity constant $\gamma >0$. Let $u\in W^{1,2}_{d^{\ast}, T}\left(\Omega;\varLambda^k\right)$  satisfy
		\begin{align}\label{omegaequation boundary estimate}
			\int_{\Omega}\langle A(x)du;d\phi\rangle =\int_{\Omega}\langle F;d\phi\rangle \quad \text{ for all } \phi\in W^{1,2}_{T}\left(\Omega;\varLambda^k\right).
		\end{align} 
		Then for every $x_0\in\partial\Omega,$ there exists  $0 < R = R \left(x_{0}, n, \Omega \right) < 1,$
		a neighborhood $W$ of $x_0$ in $\mathbb{R}^n$ and an admissible boundary coordinate system $\Phi\in \operatorname{Diff}^{2}(\overline{B_R};\overline{W})$, such that 
		\begin{align*}
			\Phi(0)=x_0,  D\Phi(0) \in \mathbb{SO}\left(n\right), \Phi(B^{+}_{R})=\Omega\cap W  \text{ and }  \Phi(\Gamma_R)=\partial\Omega\cap W, 
		\end{align*} and a constants $C, c_{1}>0,$ depending only on $x_{0}, n, k, N, \gamma, \Omega,$ such that we have 
		\begin{align}\label{u boundary estimate}
			\left[ \nabla u \right]&_{\mathrm{BMO}\left( \Phi \left( B^{+}_{R/2}\right)\right)}^{2} \notag \\&\leq C \left\lbrace \left(   \left[ \Theta^{A}\left( 2c_{1}\sigma\right)\right]^{2} + \left[ \theta \left( \sigma\right)\right]^{2}\left( \left\lVert A \right\rVert_{L^{\infty}}^{2} + 1 \right)\right) \left[ \nabla u \right]_{\mathrm{BMO}\left(\Omega\right)}^{2} + \sigma^{-n}\kappa \right\rbrace , 
		\end{align}
		for every $ 0< \sigma < R/8,$   where 
		\begin{multline*}
			\kappa := \left(\left\lVert A \right\rVert_{L^{\infty}\left( \Omega\right)}^{2} + 	\left[ A\right]_{\mathscr{L}^{2,n}_{\frac{1}{1+ \left\lvert \log r \right\rvert}}\left(\Omega\right)}^{2} + 1 \right)\left\lVert \nabla u  \right\rVert_{L^{2}\left(\Omega\right)}^{2} + \left\lVert u  \right\rVert_{L^{2}\left(\Omega\right)}^{2} \\ + \left\lVert F \right\rVert_{L^{2}\left(\Omega\right)}^{2} + \left[ F\right]_{\mathrm{BMO}\left(\Omega\right)}^{2}. 
		\end{multline*}
	\end{lemma}
	\begin{proof}
		For any $x_0\in\partial\Omega,$ we use Lemma \ref{flattening lemma} to deduce the existence a radius $ 0 < R_{0} < 1$ and a neighborhood $U$ of $x_0$ in $\mathbb{R}^n$ and an admissible boundary coordinate system $\Phi\in \operatorname{Diff}^{2}(\overline{B_{R_{0}}};\overline{U})$, satisfying 
		\begin{align*}
			\Phi(0)=x_0,  D\Phi(0) \in \mathbb{SO}\left(n\right), \Phi(B^{+}_{R_{0}})=\Omega\cap U  \text{ and }  \Phi(\Gamma_{R_{0}})=\partial\Omega\cap U, 
		\end{align*}  such that 
		$\omega=\Phi^{\ast}\left(u\right)\in W^{1,2}(B^{+}_{R_{0}};\varLambda^{k})$ satisfies 
		$e_{n}\wedge \omega = 0 $ on $\Gamma_{R_{0}}$ and  
		\begin{multline*}
			\int_{B^{+}_{R_{0}}}\langle \Tilde{A}\left( x \right)d\omega ;d\psi\rangle +\int_{B^{+}_{R_{0}}}\langle d^{\ast}\omega;d^{\ast}\psi\rangle  -\int_{B^{+}_{R_{0}}}\langle\Tilde{F} ;d\psi\rangle  \notag \\ + \int_{B_{R_{0}}^{+}} \langle \mathrm{P}\omega +  \mathrm{R}\nabla \omega ; \psi \rangle +  \int_{B^{+}_{R_{0}}}\langle \mathrm{Q} \omega ;\nabla\psi\rangle+\int_{B^{+}_{R_{0}}}\langle \mathrm{S}\nabla \omega;\nabla \psi\rangle=0  
		\end{multline*}
		for all $\psi\in W^{1,2}_{T, \text{flat}}(B^{+}_{R_{0}};\varLambda^{k}),$ where $\Tilde{A}, \Tilde{F}, \mathrm{P}, \mathrm{Q}, \mathrm{R}, \mathrm{S}$ are as in Lemma \ref{flattening lemma}. \smallskip 
		
		Now  let $0 < R < R_{0}.$ We are going to choose $R$ later ( see \eqref{choice R case 1} and \eqref{choice R case 2} ). The constants in all the estimates that we would derive from here onward may depend on $R_{0},$ but does not depend on $R.$ \smallskip

		Fix $0 < \sigma < R/16.$ Let $y = \left( y', y_{n}\right) \in B^{+}_{R/2}.$ With abuse of notation, we denote the point $\left( y', 0\right) \in \partial\mathbb{R}^{n}_{+}$ by $y'.$ Exactly one of the following can happen.  
		\begin{enumerate}[(I)]
			\item If $y_{n} > \sigma,$ then $B_{\sigma}\left( y\right) \subset \subset B^{+}_{R}.$
			\item or $0 \leq y_{n} \leq \sigma$ and then $B_{\sigma}\left( y\right) \cap B^{+}_{R} \subset B^{+}_{2\sigma}\left( y' \right) \subset B^{+}_{3R/4}. $
		\end{enumerate}
		
		\textbf{Case (I):} By Proposition \ref{existence in W012}, we can find $\alpha \in W_{0}^{1,2} \left( B\left( y, \sigma \right), \varLambda^{k}\right),$ the unique weak solution to the following Dirichlet BVP 
		\begin{multline}\label{dirichlet bvp boundary interior case}
			\int_{B_{\sigma}\left( y\right)}\left\langle \left(\Tilde{A}\right)_{B\left( y, \sigma \right)} d\alpha ; d\phi \right\rangle + \int_{B_{\sigma}\left( y\right)}\left\langle d^{\ast}\alpha ; d^{\ast}\phi\right\rangle \\= \int_{B_{\sigma}\left( y\right)}\left\langle \Tilde{F} ; d\phi\right\rangle - \int_{B_{\sigma}\left( y\right)}\left\langle G ; \nabla \phi \right\rangle - \int_{B_{\sigma}\left( y\right)}\left\langle g;\phi \right\rangle 
		\end{multline}
		for all $\phi \in W^{1,2}_{0}\left(B_{\sigma}\left( y\right), \varLambda^{k}\right)$, where $g:= \mathrm{P}\omega +  \mathrm{R}\nabla \omega $ and 
		\begin{align*}
			G &:= \mathrm{Q}\omega + \mathrm{S}\nabla \omega +  \left[ \Tilde{A}\left( x \right) - \left(\Tilde{A}\right)_{B_{\sigma}\left( y\right)}\right]d\omega := G_{1} + G_{2} + G_{3}.
		\end{align*}
		Using Proposition \ref{Gaffney with no L2 term} and \eqref{dirichlet bvp boundary interior case}, we have, 
		\begin{align}\label{dalpha esti1}
			\int_{B_{\sigma}\left( y\right)}\left\lvert \nabla \alpha \right\rvert^{2} &\leq C\left(\int_{B_{\sigma}\left( y\right)}\left\lvert d\alpha \right\rvert^{2} + \int_{B_{\sigma}\left( y\right)}\left\lvert d^{\ast} \alpha \right\rvert^{2}\right)  \notag\\
			&\leq C \left( \int_{B_{\sigma}\left( y\right)}\left\langle \left(\Tilde{A}\right)_{B_{\sigma}\left( y\right)} d\alpha  ; d\alpha \right\rangle + \int_{B_{\sigma}\left( y\right)}\left\lvert d^{\ast} \alpha \right\rvert^{2}\right) \notag \\&= C \left(\int_{B_{\sigma}\left( y\right)}\left\langle \Tilde{F}  ; d\alpha \right\rangle - \int_{B_{\sigma}\left( y\right)}\left\langle G  ; d\alpha \right\rangle - \int_{B_{\sigma}\left( y\right)}\left\langle g  ; \alpha \right\rangle \right).  
		\end{align}
		Now we estimate each term, starting with the easy ones. Using Poincar\'{e}-Sobolev, H\"{o}lder and Young's inequality with $\varepsilon>0,$ we deduce 
		\begin{align}\label{dalphaesti8}
			\left\lvert \int_{B_{\sigma}\left( y\right)}\left\langle g  ; \alpha \right\rangle\right\rvert &\leq C\left( \int_{B_{\sigma}\left( y\right)}\left\lvert g\right\rvert^{\frac{2n}{n+2}}\right)^{\frac{n+2}{2n}}\left( \int_{B_{\sigma}\left( y\right)}\left\lvert \nabla \alpha\right\rvert^{2}\right)^{\frac{1}{2}} \notag \\
			&\begin{multlined}[b]
				\leq C_{\varepsilon}\sigma^{2}\left[ \left\lVert \mathrm{P} \right\rVert_{L^{\infty}}^{2} \int_{B_{\sigma}\left( y\right)}\left\lvert \omega \right\rvert^{2} + \left\lVert \mathrm{R} \right\rVert_{L^{\infty}}^{2}\int_{B_{\sigma}\left( y\right)}\left\lvert \nabla\omega \right\rvert^{2}\right] \\ +\varepsilon \int_{B_{\sigma}\left( y\right)}\left\lvert \nabla\alpha \right\rvert^{2}.
			\end{multlined}
		\end{align}
		Similarly, we deduce 
		\begin{align}\label{dalpha esti4}
			\left\lvert \int_{B_{\sigma}\left( y\right)}\left\langle G_{1}  ; d\alpha \right\rangle\right\rvert  
			&\leq  C\varepsilon \int_{B_{\sigma}\left( y\right)}\left\lvert \nabla\alpha \right\rvert^{2} + C_{\varepsilon}\left\lVert \mathrm{Q} \right\rVert_{L^{\infty}}^{2}\int_{B_{\sigma}\left( y\right)}\left\lvert \omega\right\rvert^{2}.  
		\end{align}
		Using Young's inequality with $\varepsilon>0,$ we have 
		\begin{align}\label{dalpha esti3}
			\left\lvert \int_{B_{\sigma}\left( y\right)}\left\langle \Tilde{F}  ; d\alpha \right\rangle\right\rvert  
			&=  	\left\lvert \int_{B_{\sigma}\left( y\right)}\left\langle \Tilde{F} - \left(\Tilde{F} \right)_{B_{\sigma}\left( y\right)}  ; d\alpha \right\rangle\right\rvert  \notag \\
			&\leq  C\varepsilon \int_{B_{\sigma}\left( y\right)}\left\lvert \nabla\alpha \right\rvert^{2} + C_{\varepsilon}\int_{B_{\sigma}\left( y\right)}\left\lvert \Tilde{F} - \left(\Tilde{F} \right)_{B_{\sigma}\left( y\right)}\right\rvert^{2}. 
		\end{align}
		Now we estimate the tricky terms. By Young's inequality with $\varepsilon>0,$ we deduce  
		\begin{align*}
			\left\lvert \int_{B_{\sigma}\left( y\right)}\left\langle G_{2}  ; d\alpha \right\rangle\right\rvert  &= 	\left\lvert \int_{B_{\sigma}\left( y\right)}\left\langle G_{2} - \left(G_{2}\right)_{B_{\sigma}\left( y\right)}  ; d\alpha \right\rangle\right\rvert \\
			&\leq  C\varepsilon \int_{B_{\sigma}\left( y\right)}\left\lvert \nabla\alpha \right\rvert^{2} +  C_{\varepsilon}\int_{B_{\sigma}\left( y\right)}\left\lvert G_{2} - \left(G_{2}\right)_{B_{\sigma}\left( y\right)}\right\rvert^{2}.  
		\end{align*}
		Note that for scalar functions $f,g,$ where $f \in C^{1}$ and $g  \in \mathrm{BMO},$ we have,  
		\begin{align*}
			\int_{B_{\sigma}\left( y\right)}&\left\lvert fg - \left(fg\right)_{B_{\sigma}\left( y\right)}\right\rvert^{2} \\&\leq  	\int_{B_{\sigma}\left( y\right)}\left\lvert fg - \left(f\right)_{B_{\sigma}\left(y\right)} \left(g\right)_{B_{\sigma}\left(y\right)}\right\rvert^{2} \\
			&\leq  	\left\lVert f \right\rVert_{L^{\infty}}^{2}\int_{B_{\sigma}\left( y\right)}\left\lvert g - \left(g\right)_{B_{\sigma}\left(y\right)}\right\rvert^{2} + \left\lVert \nabla f \right\rVert_{L^{\infty}}^{2} \sigma^{n+2}\left\lvert \left(g\right)_{B_{\sigma}\left(y\right)}\right\rvert^{2}. 
		\end{align*}
		Using this for each component, and using properties of $\mathrm{S},$ we easily estimate, 
		\begin{multline*}
			\int_{B_{\sigma}\left( y\right)}\left\lvert \mathrm{S}\nabla \omega  - \left( \mathrm{S}\nabla \omega\right)_{B_{\sigma}\left( y\right)}\right\rvert^{2} \leq CR^{2}\int_{B_{\sigma}\left( y\right)}\left\lvert \nabla \omega  - \left( \nabla \omega\right)_{B_{\sigma}\left( y\right)}\right\rvert^{2}  \\ + C\sigma^{n+2}\left\lvert \left( \nabla \omega\right)_{B_{\sigma}\left( y\right)}\right\rvert^{2}.
		\end{multline*}
		Applying Proposition \ref{integral average by bmo norm interior}, we have 
		\begin{multline*}
			\int_{B_{\sigma}\left( y\right)}\left\lvert G_{2} - \left(G_{2}\right)_{B_{\sigma}\left( y\right)}\right\rvert^{2} \leq C\sigma^{n+2}\left( 1 + \left\lvert \log \sigma \right\rvert\right)^{2}\left[ \nabla \omega\right]_{\mathrm{BMO}\left(B^{+}_{R}\right)}^{2} \\+CR^{2}\int_{B_{\sigma}\left( y\right)}\left\lvert \nabla \omega  - \left( \nabla \omega\right)_{B_{\sigma}\left( y\right)}\right\rvert^{2} +\sigma^{2}\left\lVert \nabla\omega \right\rVert_{L^{2}\left( B^{+}_{R} \right)}^{2}.
		\end{multline*}
		Hence, we arrive at 
		\begin{multline}\label{dalpha esti5}
			\left\lvert \int_{B_{\sigma}\left( y\right)}\left\langle G_{2}  ; d\alpha \right\rangle\right\rvert  \leq C_{\varepsilon}\sigma^{n}\left( \theta\left(\sigma\right)\right)^{2}\left[\nabla \omega \right]_{\mathrm{BMO}\left(B^{+}_{R}\right)}^{2} + C_{\varepsilon} \sigma^{2}\left\lVert \nabla\omega \right\rVert_{L^{2}\left( B^{+}_{R} \right)}^{2} \\+ CR^{2}\int_{B_{\sigma}\left( y\right)}\left\lvert \nabla \omega  - \left( \nabla \omega\right)_{B_{\sigma}\left( y\right)}\right\rvert^{2} + C\varepsilon \int_{B_{\sigma}\left( y\right)}\left\lvert \nabla\alpha \right\rvert^{2}.
		\end{multline}
		Finally, we have
		\begin{align}\label{dalpha esti6}
			\left\lvert \int_{B_{\sigma}\left( y\right)}\left\langle G_{3}  ; d\alpha \right\rangle\right\rvert  
			\leq  C\varepsilon \int_{B_{\sigma}\left( y\right)}\left\lvert \nabla\alpha \right\rvert^{2}  + C_{\varepsilon}\int_{B_{\sigma}\left( y\right)}\left\lvert G_{3}\right\rvert^{2}.  
		\end{align}
		Now we estimate the last integral on the right in the last inequality. We have 
		\begin{align*}
			\int_{B_{\sigma}\left( y\right)}\left\lvert G_{3}\right\rvert^{2} &=  \int_{B_{\sigma}\left( y\right)}\left\lvert \left[ \Tilde{A}\left( x \right) - \left(\Tilde{A}\right)_{B_{\sigma}\left( y\right)}\right]d\omega\right\rvert^{2} \\
			&\leq \left( \int_{B_{\sigma}\left( y\right)}\left\lvert \Tilde{A}\left( x \right) - \left(\Tilde{A}\right)_{B_{\sigma}\left( y\right)}\right\rvert^{4} \right)^{\frac{1}{2}}\left( \int_{B_{\sigma}\left( y\right)}\left\lvert d\omega\right\rvert^{4} \right)^{\frac{1}{2}}.
		\end{align*}
		By Proposition \ref{Lphi majorizes Np interior}, we have 
		\begin{align*}
			\left( \int_{B_{\sigma}\left( y\right)}\left\lvert \Tilde{A}\left( x \right) - \left(\Tilde{A}\right)_{B_{\sigma}\left( y\right)}\right\rvert^{4} \right)^{\frac{1}{2}} \leq \frac{c\sigma^{\frac{n}{2}}}{\left(1 + \left\lvert \log \sigma \right\rvert \right)^{2}}\left[ \Theta^{\Tilde{A}}\left( \sigma\right)\right]^{2}.
		\end{align*}
		Moreover, using Proposition \ref{Lphi majorizes Np interior} and Proposition \ref{integral average by bmo norm interior}, we also have 
		\begin{align*}
			\left( \int_{B_{\sigma}\left( y\right)}\left\lvert d\omega\right\rvert^{4} \right)^{\frac{1}{4}} 
			&\leq  c\sigma^{\frac{n}{4}} \left[ \left( \fint_{B_{\sigma}\left( y\right)}\left\lvert d\omega - \left( d\omega\right)_{B_{\sigma}\left( y\right)}\right\rvert^{4} \right)^{\frac{1}{4}}  + \left\lvert \left( d\omega\right)_{B_{\sigma}\left( y\right)}\right\rvert \right] \\
			&\leq c\sigma^{\frac{n}{4}} \left[ \left( 1 + \left\lvert \log \sigma \right\rvert\right)\left[ d\omega\right]_{\mathrm{BMO}\left(B^{+}_{R}\right)} + \sigma^{-\frac{n}{2}}\left\lVert d\omega \right\rVert_{L^{2}\left( B^{+}_{R} \right)}\right]. 
		\end{align*}
		Combining the estimates, we finally obtain 
		\begin{multline}\label{dalpha esti7}
			\int_{B_{\sigma}\left( y\right)}\left\lvert G_{3}\right\rvert^{2} \leq \frac{c}{\left(1 + \left\lvert \log \sigma \right\rvert \right)^{2}}\left[ \tilde{A}\right]_{\mathscr{L}^{2,n}_{\frac{1}{1+ \left\lvert \log r \right\rvert}}\left(B^{+}_{R}\right)}^{2}\left\lVert \nabla \omega \right\rVert_{L^{2}\left( B^{+}_{R} \right)}^{2} \\+ c \sigma^{n}  \left[ \Theta^{\Tilde{A}}\left( \sigma\right)\right]^{2}\left[ \nabla\omega\right]_{\mathrm{BMO}\left(B^{+}_{R}\right)}^{2}. 
		\end{multline}
		Thus, combining the estimates  \eqref{dalpha esti1}, \eqref{dalphaesti8}, \eqref{dalpha esti3}, \eqref{dalpha esti4}, \eqref{dalpha esti5}, \eqref{dalpha esti6}, \eqref{dalpha esti7} and choosing $\varepsilon>0$ small enough, we deduce 
		\begin{multline}\label{dalpha final esti}
			\int_{B_{\sigma}\left( y\right)}\left\lvert \nabla \alpha \right\rvert^{2} \leq c \sigma^{n}\left( \sigma^{2}\left( 1 + \left\lvert \log \sigma \right\rvert\right)^{2} + c\left[\Theta^{\Tilde{A}}\left( \sigma\right)\right]^{2} \right) \left[\nabla \omega \right]_{\mathrm{BMO}\left(B^{+}_{R}\right)}^{2} \\+CR^{2}\int_{B_{\sigma}\left( y\right)}\left\lvert \nabla \omega  - \left( \nabla \omega\right)_{B_{\sigma}\left( y\right)}\right\rvert^{2}   + \sigma^{n}\kappa_{1}, 
		\end{multline}
		where 
		\begin{multline*}
			\kappa_{1}:= \fint_{B_{\sigma}\left( y\right)}\left\lvert \Tilde{F} - \left(\Tilde{F} \right)_{B_{\sigma}\left( y\right)}\right\rvert^{2} + \sigma^{-n}\left(\left\lVert \mathrm{Q} \right\rVert_{L^{\infty}}^{2} + \left\lVert \mathrm{P} \right\rVert_{L^{\infty}}^{2}\right)\left\lVert \omega \right\rVert_{L^{2}\left( B^{+}_{R} \right)}^{2} \\ + \sigma^{-n}\left( 1 + \left\lVert \mathrm{R} \right\rVert_{L^{\infty}}^{2}+\left[ \tilde{A}\right]_{\mathscr{L}^{2,n}_{\frac{1}{1+ \left\lvert \log r \right\rvert}}\left(B^{+}_{R}\right)}^{2} \right) \left\lVert \nabla\omega \right\rVert_{L^{2}\left( B^{+}_{R} \right)}^{2}.
		\end{multline*}
		Now, $\beta = \omega -\alpha,$ satisfies, for all $\phi \in W^{1,2}_{0}\left(B_{\sigma}\left( y\right), \varLambda^{k}\right),$ 		\begin{align*}
			\int_{B_{\sigma}\left( y\right)}\left\langle \left(\Tilde{A}\right)_{B\left( y, \sigma \right)} d\beta ; d\phi \right\rangle + \int_{B_{\sigma}\left( y\right)}\left\langle d^{\ast}\beta ; d^{\ast}\phi\right\rangle = 0.  
		\end{align*}
		Standard arguments, using decay estimates for $\beta$ implies, for any $0 < \rho < \sigma ,$ 
		\begin{align*}
			\int_{B_{\rho}\left( y\right)} &\left\lvert \nabla \omega - \left( \nabla \omega \right)_{B_{\rho}\left( y\right)}\right\rvert^{2} \\
			&\leq c \left(\frac{\rho}{\sigma}\right)^{n+2} 	\int_{B_{\sigma}\left( y\right)} \left\lvert \nabla \omega - \left( \nabla \omega \right)_{B_{\sigma}\left( y\right)}\right\rvert^{2} + c\int_{B_{\sigma}\left( y\right)} \left\lvert \nabla \alpha \right\rvert^{2}. 
		\end{align*}
		Defining $	\displaystyle	\psi \left( r \right):= \int_{B_{r}\left( y\right)} \left\lvert \nabla \omega - \left( \nabla \omega \right)_{B_{r}\left( y\right)}\right\rvert^{2} $,  and using \eqref{dalpha final esti}, we have 
		\begin{multline}\label{preiteration lemma}
			\psi\left(\rho\right) \leq c\left[ \left(\frac{\rho}{\sigma}\right)^{n+2} + CR^{2} \right] \psi \left(\sigma \right) \\ + C\sigma^{n}\left( \sigma^{2}\left( 1 + \left\lvert \log \sigma \right\rvert\right)^{2} +\left[\Theta^{\Tilde{A}}\left( \sigma\right)\right]^{2} \right) \left[\nabla \omega \right]_{\mathrm{BMO}\left(B^{+}_{R}\right)}^{2} + C\sigma^{n}\kappa_{1}.
		\end{multline}
		
		Now we choose $0 < R < 1$ in \eqref{preiteration lemma} such that
		\begin{align}\label{choice R case 1}
			CR^{2} < \varepsilon_{0},
		\end{align}
		where $\varepsilon_{0}>0 $ is the smallness parameter given by the standard iteration lemma ( Lemma 5.13 in \cite{giaquinta-martinazzi-regularity} ). Thus, the iteration lemma implies  
		\begin{align}\label{rho small boundary interior case final}
			\frac{1}{\rho^{n}}\psi \left(\rho\right) \leq C\left[ \frac{1}{\sigma^{n}}\psi\left(\sigma\right) + \left( \theta\left( \sigma\right)^{2} +\left[\Theta^{\Tilde{A}}\left( \sigma\right)\right]^{2} \right) \left[\nabla \omega \right]_{\mathrm{BMO}\left(B^{+}_{R}\right)}^{2} + \kappa_{1} \right] 
		\end{align}
		for every $0 < \rho < \sigma.$ \smallskip 
		
		\textbf{Case II: } For this case, we begin by claiming that there exists $\beta \in \omega + W_{T}^{1,2}\left( B^{+}_{2\sigma}\left( y' \right); \varLambda^{k}\right)$  which is the unique minimizer of the minimization problem 
		\begin{align*}
			m:= \inf \left\lbrace  I\left[u\right]: u \in \omega + W_{T}^{1,2}\left( B^{+}_{2\sigma}\left( y' \right); \varLambda^{k}\right)\right\rbrace,  
		\end{align*}
		where 
		\begin{align*}
			 I\left[u\right]: = \frac{1}{2}\int_{B^{+}_{2\rho}\left( y' \right)} \left[ \left\langle \left(\Tilde{A}\right)_{B^{+}_{2\sigma}\left( y' \right)} du ; du \right\rangle + \left\lvert d^{\ast}u\right\rvert^{2}\right]. 
		\end{align*}
		First note that although $B^{+}_{2\sigma}\left( y' \right)$ is not $C^{2},$ but it is a convex Lipschitz domain. Thus, the Gaffney inequality in Proposition \ref{Gaffney with no L2 term} holds in $B^{+}_{2\sigma}\left( y' \right).$ This can be proved by a approximation argument ( see \cite{Mitrea_Gaffneyineq} ). Thus, if $\left\lbrace \beta_{s} \right\rbrace_{s \in \mathbb{N}}$ is a minimizing sequence, using Legendre-Hadamard condition for $A$, we immediately deduce 
		\begin{align*}
			\int_{B^{+}_{2\rho}\left( y' \right)} \left( \left\lvert d\beta_{s}\right\rvert^{2} + \left\lvert d^{\ast}\beta_{s}\right\rvert^{2}\right) &\leq c I\left[ \beta_{s}\right] \leq c\left(m+1\right)
		\end{align*}
		for all but finitely many $s \in \mathbb{N}.$ Now, since $\beta_{s} -\omega \in W_{T}^{1,2}\left( B^{+}_{2\sigma}\left( y' \right); \varLambda^{k}\right)$ for all $s \in \mathbb{N},$ using the Gaffney inequality and a simple contradiction argument ( see Proposition $1$ in \cite{Sil_nonlinearStein} ), we arrive at 
		\begin{align*}
			\left\lVert \beta_{s} \right\rVert&_{W^{1,2}\left( B^{+}_{2\rho}\left( y' \right); \varLambda^{k}\right)}^{2} \\&\leq \left\lVert \beta_{s} -\omega \right\rVert_{W^{1,2}\left( B^{+}_{2\rho}\left( y' \right); \varLambda^{k}\right)}^{2} + \left\lVert \omega \right\rVert_{W^{1,2}\left( B^{+}_{2\rho}\left( y' \right); \varLambda^{k}\right)}^{2} \\&\leq c\int_{B^{+}_{2\rho}} \left\lvert \nabla \beta_{s} - \nabla \omega \right\rvert^{2} + \left\lVert \omega \right\rVert_{W^{1,2}\left( B^{+}_{2\rho}\left( y' \right); \varLambda^{k}\right)}^{2} \\&\leq c  \int_{B^{+}_{2\rho}\left( y' \right)} \left( \left\lvert d\beta_{s} - d\omega\right\rvert^{2} + \left\lvert d^{\ast}\beta_{s} - d^{\ast}\omega\right\rvert^{2}\right) + \left\lVert \omega \right\rVert_{W^{1,2}\left( B^{+}_{2\rho}\left( y' \right); \varLambda^{k}\right)}^{2}  \\&\leq c  \int_{B^{+}_{2\rho}\left( y' \right)} \left( \left\lvert d\beta_{s} \right\rvert^{2} + \left\lvert d^{\ast}\beta_{s} \right\rvert^{2}\right) + c\left\lVert \omega \right\rVert_{W^{1,2}\left( B^{+}_{2\rho}\left( y' \right); \varLambda^{k}\right)}^{2}  \\&\leq c\left(m+1\right) +c \left\lVert \omega \right\rVert_{W^{1,2}\left( B^{+}_{2\rho}\left( y' \right); \varLambda^{k}\right)}^{2} 
		\end{align*} 
		for all $s \in \mathbb{N}.$ Thus, $\left\lbrace \beta_{s} \right\rbrace_{s \in \mathbb{N}}$ is uniformly bounded in $W^{1,2}$ and thus, up to the extraction of a subsequence which we do not relabel, we have 
		\begin{align*}
			\beta_{s} \rightharpoonup \beta \qquad \text{ in } W^{1,2}\left( B^{+}_{2\rho}\left( y' \right); \varLambda^{k}\right)
		\end{align*}
		for some $\beta \in W^{1,2}\left( B^{+}_{2\rho}\left( y' \right); \varLambda^{k}\right).$ Moreover, since $W^{1,2}_{T}\left( B^{+}_{2\rho}\left( y' \right); \varLambda^{k}\right)$, being a linear subspace is a convex subset of $W^{1,2}\left( B^{+}_{2\rho}\left( y' \right); \varLambda^{k}\right), $ we also have $\beta \in \omega + W^{1,2}_{T}\left( B^{+}_{2\rho}\left( y' \right); \varLambda^{k}\right).$ Now it is easy to check that $\beta$ is a minimizer and by strict convexity, the unique minimizer. \smallskip 
		
		Now writing down the Euler-Lagrange equation for the minimization problem, we immediately deduce that  $\beta \in \omega + W^{1,2}_{T}\left( B^{+}_{2\rho}\left( y' \right); \varLambda^{k}\right)$ satisfies \begin{align*}
			\int_{B^{+}_{2\sigma}\left( y' \right)} \left\langle \left(\Tilde{A}\right)_{B^{+}_{2\sigma}\left( y' \right)} d\beta ; d\phi \right\rangle + \int_{B^{+}_{2\sigma}\left( y' \right)} \left\langle d^{\ast}\beta; d^{\ast}\phi \right\rangle  
			= 0  
		\end{align*}
		for all $\phi \in W_{T}^{1,2} \left( B^{+}_{2\sigma}\left( y' \right); \varLambda^{k}\right).$ Thus, arguing exactly as in Theorem 1, Theorem 2 and Theorem 3 of \cite{Sil_linearregularity}, for any $0 < \rho < \sigma ,$ we have the following decay estimate 
		\begin{align}\label{boundary decay for beta}
			\int_{B^{+}_{2\rho}\left( y' \right)} \left\lvert \nabla \beta - \left( \nabla \beta \right)_{B^{+}_{2\rho}\left( y' \right)}\right\rvert^{2}
			\leq c \left(\frac{\rho}{\sigma}\right)^{n+2} 	\int_{B^{+}_{2\sigma}\left( y' \right))} \left\lvert \nabla \beta - \left( \nabla \beta \right)_{B^{+}_{2\sigma}\left( y' \right)}\right\rvert^{2}. 
		\end{align} 
		
		Now, setting $\alpha = \omega -\beta,$ we deduce that $\alpha \in W^{1,2}_{T}\left( B^{+}_{2\rho}\left( y' \right); \varLambda^{k}\right)$ satisfies 
		\begin{multline*}
			\int_{B^{+}_{2\sigma}\left( y' \right)}\left\langle \left(\Tilde{A}\right)_{B^{+}_{2\sigma}\left( y' \right)} d\alpha ; d\phi \right\rangle + \int_{B^{+}_{2\sigma}\left( y' \right)}\left\langle d^{\ast}\alpha ; d^{\ast}\phi\right\rangle \\= \int_{B^{+}_{2\sigma}\left( y' \right)}\left\langle \Tilde{F} ; d\phi\right\rangle - \int_{B^{+}_{2\sigma}\left( y' \right)}\left\langle G ; \nabla \phi \right\rangle - \int_{B^{+}_{2\sigma}\left( y' \right)}\left\langle g;\phi \right\rangle ,
		\end{multline*}
		for all $\phi \in W_{T}^{1,2} \left( B^{+}_{2\sigma}\left( y' \right); \varLambda^{k}\right),$ where $g:= \mathrm{P}\omega +  \mathrm{R}\nabla \omega $ and 
		\begin{align*}
			G &:= \mathrm{Q}\omega + \mathrm{S}\nabla \omega +  \left[ \Tilde{A}\left( x \right) - \left(\Tilde{A}\right)_{B^{+}_{2\sigma}\left( y' \right)}\right]d\omega.
		\end{align*}
Arguing exactly as before, but using Proposition \ref{Lphi majorizes Np boundary} and Proposition \ref{integral average by bmo norm boundary} in place of Proposition \ref{Lphi majorizes Np interior} and Proposition \ref{integral average by bmo norm interior}, respectively, we deduce the estimate 
		\begin{multline}\label{dalpha esti boundary boundary case}
			\int_{B^{+}_{2\sigma}\left( y' \right)}\left\lvert \nabla \alpha \right\rvert^{2}  \leq c\sigma^{n}\left( \sigma^{2}\left( 1 + \left\lvert \log \sigma \right\rvert\right)^{2} + c\left[\Theta^{\Tilde{A}}\left( \sigma\right)\right]^{2} \right) \left[\nabla \omega \right]_{\mathrm{BMO}\left(B^{+}_{R}\right)}^{2} \\+CR^{2}\int_{B^{+}_{2\sigma}\left( y' \right)}\left\lvert \nabla \omega  - \left( \nabla \omega\right)_{B^{+}_{2\sigma}\left( y' \right)}\right\rvert^{2} + c\sigma^{n}\kappa_{2},
		\end{multline}
		where 
		\begin{multline*}
			\kappa_{2}:= \fint_{B^{+}_{2\sigma}\left( y' \right)}\left\lvert \Tilde{F} - \left(\Tilde{F} \right)_{B^{+}_{2\sigma}\left( y' \right)}\right\rvert^{2} + \sigma^{-n}\left(\left\lVert \mathrm{Q} \right\rVert_{L^{\infty}}^{2} + \left\lVert \mathrm{P} \right\rVert_{L^{\infty}}^{2}\right)\left\lVert \omega \right\rVert_{L^{2}\left( B^{+}_{R} \right)}^{2} \\ + \sigma^{-n}\left( 1 + \left\lVert \mathrm{R} \right\rVert_{L^{\infty}}^{2}+\left[ \tilde{A}\right]_{\mathscr{L}^{2,n}_{\frac{1}{1+ \left\lvert \log r \right\rvert}}\left(B^{+}_{R}\right)}^{2} \right) \left\lVert \nabla\omega \right\rVert_{L^{2}\left( B^{+}_{R} \right)}^{2}.
		\end{multline*}
		Now by standard arguments, \eqref{boundary decay for beta} and \eqref{dalpha esti boundary boundary case} and setting $$\displaystyle \psi \left( r \right):= \int_{B^{+}_{2r}\left( y' \right)} \left\lvert \nabla \omega - \left( \nabla \omega \right)_{B^{+}_{2r}\left( y' \right)}\right\rvert^{2},$$  we deduce 
		\begin{multline}\label{preiteration lemma boundary case}
			\psi\left(\rho\right) \leq c\left[ \left(\frac{\rho}{\sigma}\right)^{n+2} + CR^{2} \right] \psi \left(\sigma \right) \\ + C\sigma^{n}\left( \theta\left( \sigma\right)^{2}  +\left[\Theta^{\Tilde{A}}\left( \sigma\right)\right]^{2} \right) \left[\nabla \omega \right]_{\mathrm{BMO}\left(B^{+}_{R}\right)}^{2} + C\sigma^{n}\kappa_{2}.
		\end{multline}
		Now we choose $0 < R < 1$ in \eqref{preiteration lemma boundary case} such that
		\begin{align}\label{choice R case 2}
			CR^{2} < \varepsilon_{0},
		\end{align}
		where $\varepsilon_{0}>0$ is the smallness parameter given by the standard iteration lemma ( Lemma 5.13 in \cite{giaquinta-martinazzi-regularity} ). Thus, using the iteration lemma and the fact that $B_{\rho}\left( y\right)\cap B_{R}^{+} \subset B^{+}_{2\rho}\left( y' \right),$ we obtain  
		\begin{multline}\label{rho small boundary boundary case final}
			\frac{1}{\rho^{n}}\int_{B_{\rho}\left( y\right)\cap B_{R}^{+}} \left\lvert \nabla \omega - \left( \nabla \omega \right)_{B_{\rho}\left( y\right)\cap B_{R}^{+}}\right\rvert^{2}  \\\leq C\left[ \frac{1}{\sigma^{n}}\psi\left(\sigma\right) + \left( \theta\left( \sigma\right)^{2} +\left[\Theta^{\Tilde{A}}\left( \sigma\right)\right]^{2} \right) \left[\nabla \omega \right]_{\mathrm{BMO}\left(B^{+}_{R}\right)}^{2} + \kappa_{3} \right] 
		\end{multline}
		On the other hand, clearly we have 
		\begin{align*}
			\frac{1}{\rho^{n}} \int_{B^{+}_{R} \cap B\left( y, \rho \right)} \left\lvert \nabla \omega - \left( \nabla \omega \right)_{B^{+}_{R} \cap B\left( y, \rho \right)}\right\rvert^{2} \leq C\sigma^{-n}\left\lVert \nabla\omega \right\rVert_{L^{2}\left( B^{+}_{R} \right)}^{2}, 
		\end{align*}
		for any $\rho \geq \sigma$ and any $y \in B^{+}_{R/2}.$ Combining this with \eqref{rho small boundary interior case final} and \eqref{rho small boundary boundary case final} and taking supremum over all $\rho >0$ and $y \in B^{+}_{R/2},$ we arrive at 
		\begin{align*}
			\left[ \nabla \omega \right]_{\mathrm{BMO}\left(B^{+}_{R/2}\right)}^{2} \leq C \left(  \theta\left( \sigma\right)^{2}  +\left[\Theta_{\frac{1}{1+ \left\lvert \log r \right\rvert}}^{\Tilde{A}}\left( \sigma\right)\right]^{2} \right) \left[\nabla \omega \right]_{\mathrm{BMO}\left(B^{+}_{R}\right)}^{2} + \sigma^{-n}\tilde{\kappa},
		\end{align*} 
		where 
		\begin{align*}
			\tilde{\kappa}:= \left[ \Tilde{F} \right]_{\mathrm{BMO}\left(B^{+}_{R}\right)}^{2} &+ \sigma^{-n}\left(\left\lVert \mathrm{Q} \right\rVert_{L^{\infty}}^{2} + \left\lVert \mathrm{P} \right\rVert_{L^{\infty}}^{2}\right)\left\lVert \omega \right\rVert_{L^{2}\left( B^{+}_{R} \right)}^{2} \\ &+ \sigma^{-n}\left( 1 + \left\lVert \mathrm{R} \right\rVert_{L^{\infty}}^{2}+\left[ \tilde{A}\right]_{\mathscr{L}^{2,n}_{\frac{1}{1+ \left\lvert \log r \right\rvert}}\left(B^{+}_{R}\right)}^{2} \right) \left\lVert \nabla\omega \right\rVert_{L^{2}\left( B^{+}_{R} \right)}^{2}.
		\end{align*}
		Since $u = \left( \Phi^{-1}\right)^{\ast}\omega,$ \eqref{u boundary estimate} follows and this finishes the proof.  	 
	\end{proof}
	
	\subsection{Approximation}	 
	\begin{lemma}[Approximation lemma]\label{approximation lemma}
		If Theorem \ref{Main theorem BMO linear A only} holds with the additional assumption that $A$ is smooth, then Theorem \ref{Main theorem BMO linear A only} holds. 
	\end{lemma}
	\begin{proof}
		Let $A \in L^{\infty}\cap \mathrm{V}\mathscr{L}^{2}_{\frac{1}{1+ \left\lvert \log r \right\rvert}}\left( \Omega; \mathcal{L}\left( \varLambda^{k+1}; \varLambda^{k+1}\right)\right)$ be uniformly Legendre-elliptic with ellipticity constant $\gamma >0$ and let $F \in \mathrm{BMO}\left( \Omega; \varLambda^{k+1}\right)$. By Remark \ref{Remark about Mean Oscillation space properties} (i) , there exists a sequence $\left\lbrace A_{s}\right\rbrace_{s \in \mathbb{N}} \subset C^{\infty}(\overline{\Omega}, \mathcal{L}\left( \varLambda^{k+1}; \varLambda^{k+1}\right)$ such that 
		$A_{s}$ is uniformly Legendre-elliptic with constant $\gamma/2,$ and 
		we have the strong convergences \begin{align*}
			A_{s} \rightarrow A \qquad \text{ in } \mathscr{L}^{2,n}_{\frac{1}{1+ \left\lvert \log r \right\rvert}} \quad \text{ and } \quad \text{ in }  L^{p} \text{ for every } 1 \leq p < \infty, 
		\end{align*} along with the uniform estimates $\left\lVert A_{s}\right\rVert_{L^{\infty}} \leq \left\lVert A \right\rVert_{L^{\infty}},$ and  \begin{align*}
			\Theta^{A_{s}}\left( \rho \right) \leq c \left( \Theta^{A}\left( \rho \right) + \rho^{n}\left\lVert A \right\rVert_{L^{\infty}}\right)  \qquad \text{ for all } s \in \mathbb{N}.
		\end{align*}
		For every $s \in \mathbb{N},$ using Theorem 9 and Remark 12 in \cite{Sil_linearregularity}, we can find $u_{s} \in W^{1,2}_{d^{\ast}, T}\cap \left( \mathcal{H}^{h}_{T}\right)^{\perp}$, the unique weak solution for the following 
		\begin{align}\label{weak formulation for sequence As}
			\left\lbrace \begin{aligned}
				d^{\ast} \left( A_{s}\left( x \right) du_{s} \right)  &= d^{\ast} F &&\text{ in } \Omega, \\
				d^{\ast} u_{s} &= 0 &&\text{  in } \Omega, \\
				\nu\wedge u_{s} &= 0 &&\text{ on } \partial\Omega.
			\end{aligned}	\right. 
		\end{align}
		Putting $\phi = u_{s}$ in the weak formulation, using Young's inequality with $\varepsilon>0$ small enough and Proposition \ref{poincare gaffney}, we deduce 
		\begin{align*}
			\left\lVert u_{s} \right\rVert_{W^{1,2}\left(\Omega\right)}^{2} \leq C\left\lVert F \right\rVert_{L^{2}\left(\Omega\right)}^{2} \qquad \text{ for every } s \in \mathbb{N}. 
		\end{align*}
		Thus, up to the extraction of a subsequence,  we have
		\begin{align*}
			u_{s} \rightharpoonup u \qquad \text{ in } W^{1,2}_{d^{\ast}, T}\left( \Omega ; \varLambda^{k}\right) \qquad \text{ for some } u \in W^{1,2}_{d^{\ast}, T}\cap \left(\mathcal{H}^{k}_{T}\right)^{\perp}.
		\end{align*}
		Now, for any $1 \leq q <2,$ we have 
		\begin{align}\label{convergence in Lq}
			\left\lVert \left[ A_{s}\left(x \right) - A\left( x\right)\right]du_{s}\right\rVert_{L^{q}\left(\Omega\right)} 
			\leq  \left\lVert A_{s} - A\right\rVert_{L^{\frac{2q}{2-q}}\left(\Omega\right)}\left\lVert du_{s}\right\rVert_{L^{2}\left(\Omega\right)}\rightarrow 0. 
		\end{align}
		Now by Lemma \ref{weak formulation change linear lemma}, \eqref{weak formulation for sequence As} implies that in particular, we also have 
		\begin{align}\label{weak formulation for sequence As changed to ccinfty}
			\int_{\Omega} \left\langle A_{s}\left(x \right)du_{s}; d\eta \right\rangle = \int_{\Omega} \left\langle F; d\eta \right\rangle \qquad \text{ for every } \eta \in C_{c}^{\infty}\left( \Omega ; \varLambda^{k}\right). 
		\end{align}
		Thus, since $du_{s} \rightharpoonup du$ in $L^{2},$ in view of \eqref{convergence in Lq}, for any $\eta \in C_{c}^{\infty}\left( \Omega ; \varLambda^{k}\right),$ we deduce 
		\begin{align*}
			\int_{\Omega} &\left\langle A\left(x \right)du; d\eta \right\rangle - \int_{\Omega} \left\langle F; d\eta \right\rangle \\
			&=\int_{\Omega} \left\langle A\left(x \right)\left[ du - du_{s}\right]; d\eta \right\rangle - \int_{\Omega} \left\langle \left[ A_{s}\left(x \right) - A\left( x\right)\right]du_{s}; d\eta \right\rangle \rightarrow 0. 
		\end{align*}
		By density and Lemma \ref{weak formulation change linear lemma}, this implies $u \in W^{1,2}_{d^{\ast}, T}\cap \left( \mathcal{H}^{k}_{T}\right)^{\perp}$ is the unique weak solution of 
		\begin{align}\label{u weak solution claim}
			\left\lbrace \begin{aligned}
				d^{\ast} \left( A\left( x \right) du \right)  &= d^{\ast} F &&\text{ in } \Omega, \\
				d^{\ast} u &= 0 &&\text{  in } \Omega, \\
				\nu\wedge u &= 0 &&\text{ on } \partial\Omega,
			\end{aligned}	\right. 
		\end{align} 
		From \eqref{weak formulation for sequence As} and \eqref{u weak solution claim},  for every $\phi \in W^{1,2}_{d^{\ast}, T},$ we have 
		\begin{align*}
			\int_{\Omega} \left\langle A_{s}\left(x \right)\left[ du_{s} -du \right]; d\phi \right\rangle = 	\int_{\Omega} \left\langle \left[ A\left(x\right) - A_{s}\left(x \right)\right] du ; d\phi \right\rangle 
		\end{align*}
		Putting $\phi = u_{s} -u,$ using Young's inequality with $\varepsilon>0$ small enough and by Proposition \ref{poincare gaffney} we deduce 
		\begin{align}\label{us-u strong convergence estimate}
			\left\lVert u_{s} -u \right\rVert_{W^{1,2}\left(\Omega\right)}^{2} \leq C\left\lVert \left[ A\left(x\right) - A_{s}\left(x \right)\right] du \right\rVert_{L^{2}\left(\Omega\right)}^{2} \qquad \text{ for every } s \in \mathbb{N}. 
		\end{align}
		But since  we have 
		\begin{align*}
			\left\lVert \left[ A_{s}\left(x \right) - A\left( x\right)\right]du\right\rVert_{L^{q}\left(\Omega\right)} 
			\leq  \left\lVert A_{s} - A\right\rVert_{L^{\frac{2q}{2-q}}\left(\Omega\right)}\left\lVert du\right\rVert_{L^{2}\left(\Omega\right)}\rightarrow 0, 
		\end{align*}  for any $1\leq q<2$, RHS of \eqref{us-u strong convergence estimate} converges to $0$ by dominated convergence. Now if Theorem \ref{Main theorem BMO linear A only} holds for smooth coefficients, we have the estimates
		\begin{align*}
			\left[ \nabla u_{s}\right]_{\mathrm{BMO}\left(\Omega\right)} &\leq   C \left(  \left[ F\right]_{\mathrm{BMO}\left( \Omega\right)} + \left\lVert F \right\rVert_{L^{2}\left(\Omega\right)} + \left\lVert u_{s} \right\rVert_{L^{2}\left(\Omega\right)}\right) \qquad \text{ for all } s \in \mathbb{N}.
		\end{align*}
		Now fix $x \in \Omega$ and $r>0.$ We have 
		\begin{align*}
			&\fint_{\Omega_{r}\left(x\right)} 
			\left\lvert \nabla u - \left( \nabla u\right)_{\Omega_{r}\left(x\right)}\right\rvert^{2} \\
			&\leq 2\left[ \nabla u_{s}\right]_{\mathrm{BMO}\left(\Omega\right)}^{2} + 2\left\lvert \Omega_{r}\left(x\right)\right\rvert \int_{\Omega_{r}\left(x\right)}\left\lvert \nabla u - \nabla u_{s}\right\rvert^{2} \\
			&\leq C \left(  \left[ F\right]_{\mathrm{BMO}\left( \Omega\right)}^{2} + \left\lVert F \right\rVert_{L^{2}\left(\Omega\right)}^{2} + \left\lVert u_{s} \right\rVert_{L^{2}\left(\Omega\right)}\right)^{2} + 2\left\lvert \Omega_{r}\left(x\right)\right\rvert \int_{\Omega_{r}\left(x\right)}\left\lvert \nabla u - \nabla u_{s}\right\rvert^{2}
		\end{align*}
		for every $s \in \mathbb{N}.$ Since $u_{s} \rightarrow u$ in $W^{1,2},$ letting $s \rightarrow \infty$ and taking supremum, 		\begin{align*}
			\left[ \nabla u\right]_{\mathrm{BMO}\left(\Omega\right)} &\leq   C \left(  \left[ F\right]_{\mathrm{BMO}\left( \Omega\right)} + \left\lVert F \right\rVert_{L^{2}\left(\Omega\right)} + \left\lVert u \right\rVert_{L^{2}\left(\Omega\right)}\right).
		\end{align*}
		The estimate for $\left[ du\right]_{\mathrm{BMO}\left(\Omega\right)} $ follows similarly. \end{proof}	 
	
	\subsection{Flattening the boundary}
	\begin{lemma}[Flattening lemma]\label{flattening lemma}    
		Let $\partial\Omega$ be $C^{2}$ and let $F \in L^{2}\left(\Omega;\varLambda^{k+1}\right).$ Let 
		$A \in L^{\infty}\cap \mathrm{V}\mathscr{L}^{2,n}_{\frac{1}{1+ \left\lvert \log r \right\rvert}} \left( \Omega; \mathcal{L}\left(\varLambda^{k+1}; \varLambda^{k+1} \right)\right)$ be uniformly Legendre elliptic with ellipticity constant $\gamma$. Let $\omega\in W^{1,2}_{d^{\ast}, T}\left(\Omega;\varLambda^k\right)$  satisfy
		\begin{align}\label{omegaequation}
			\int_{\Omega}\langle A(x)d\omega;d\phi\rangle =\int_{\Omega}\langle F;d\phi\rangle \quad \text{ for all } \phi\in W^{1,2}_{T}\left(\Omega;\varLambda^k\right).
		\end{align} 
		Let $x_0\in\partial\Omega.$  Then there exists a neighborhood $U$ of $x_0$ in $\mathbb{R}^n$ and a positive number $0 < R_{0}<1,$ such that there exists an admissible boundary coordinate system $\Phi\in \operatorname{Diff}^{2}(\overline{B_{R_{0}}};\overline{U})$, such that 
		\begin{align*}
			\Phi(0)=x_0, \quad D\Phi(0) \in \mathbb{SO}\left(n\right), \quad\Phi(B^{+}_{R_{0}})=\Omega\cap U, \quad  \Phi(\Gamma_{R_{0}})=\partial\Omega\cap U, 
		\end{align*} and $u=\Phi^{\ast}\left(\omega\right)\in W^{1,2}(B^{+}_{R_{0}};\varLambda^{k})$ satisfies 
		$e_{n}\wedge u = 0$ on $\Gamma_{R_{0}},$ and 
		\begin{multline}\label{uequation}
			\int_{B^{+}_{R_{0}}}\langle \Tilde{A}\left( x \right)du ;d\psi\rangle +\int_{B^{+}_{R_{0}}}\langle d^{\ast}u;d^{\ast}\psi\rangle  -\int_{B^{+}_{R_{0}}}\langle\Tilde{F} ;d\psi\rangle   \\ + \int_{B_{R_{0}}^{+}} \langle \mathrm{P}u +  \mathrm{R}\nabla u ; \psi \rangle +  \int_{B^{+}_{R_{0}}}\langle \mathrm{Q} u ;\nabla\psi\rangle+\int_{B^{+}_{R_{0}}}\langle \mathrm{S}\nabla u ;\nabla \psi\rangle=0  
		\end{multline}
		for all $\psi\in W^{1,2}_{T, \text{flat}}(B^{+}_{R_{0}};\varLambda^{k}),$ where $\tilde{A} \in L^{\infty}\cap \mathrm{V}\mathscr{L}^{2,n}_{\frac{1}{1+ \left\lvert \log r \right\rvert}} \left( B^{+}_{R_{0}}; \mathcal{L}\left(\varLambda^{k+1}; \varLambda^{k+1} \right)\right)$ is uniformly Legendre-elliptic, $\tilde{F} \in L^{2}\left( B^{+}_{R_{0}};\varLambda^{k+1}\right),$  
		\begin{align*}
			& \mathrm{P} \in C\left(\overline{B^{+}_{R_{0}}}; \mathcal{L}\left( \varLambda^{k}; \varLambda^{k}\right)\right), \mathrm{Q} \in C\left(\overline{B^{+}_{R_{0}}}; \mathcal{L}\left( \varLambda^{k}; \varLambda^{k}\otimes \mathbb{R}^{n}\right)\right), \\& \mathrm{R} \in C\left(\overline{B^{+}_{R_{0}}}; \mathcal{L}\left( \varLambda^{k}\otimes \mathbb{R}^{n}; \varLambda^{k}\right)\right) \text{ and } \mathrm{S} \in C^{1}\left(\overline{B^{+}_{R_{0}}}; \mathcal{L}\left( \varLambda^{k}\otimes \mathbb{R}^{n}; \varLambda^{k}\otimes \mathbb{R}^{n}\right)\right).
		\end{align*}
		Furthermore, there exist constants $0 < c_{0}<1, $ $c_{1}, c_{2} >0$ and $C>0,$  all depending only on $x_{0}, n, k, N, \Omega, R_{0}$,  
		such that 
		\begin{align*}
			\left[ \tilde{A}\right]_{\mathscr{L}^{2,n}_{\frac{1}{1+ \left\lvert \log r \right\rvert}}} &\leq C \left[ A \right]_{\mathscr{L}^{2,n}_{\frac{1}{1+ \left\lvert \log r \right\rvert}}},\\ \left\lVert \tilde{A} \right\rVert_{L^{\infty}} &\leq C \left\lVert A \right\rVert_{L^{\infty}}, \\ \left\lVert \tilde{F} \right\rVert_{L^{2}} &\leq C \left\lVert F \right\rVert_{L^{2}}, \\
			\left\lVert \mathrm{P} \right\rVert_{L^{\infty}}, \left\lVert \mathrm{Q} \right\rVert_{L^{\infty}}, \left\lVert \mathrm{R} \right\rVert_{L^{\infty}} &\leq C, \\ \left\lVert \mathrm{S} \right\rVert_{L^{\infty}\left(B_{r}^{+}\right)} &\leq Cr \qquad \text{ for all } 0 < r \leq R_{0},  
		\end{align*}
		and
		\begin{align*}
			\Theta^{\tilde{A}} \left(\rho\right) &\leq c_{1} \left( \Theta^{A} \left(c_{2}\rho\right) + \theta\left(\rho\right)\right) 
		\end{align*}
		for all $ 0< \rho < \min \left\lbrace R_{0}, R_{0}/c_{2}\right\rbrace.$	Moreover, if in addition, $F \in \mathrm{BMO}\left(\Omega;\varLambda^{k+1}\right),$ then $\tilde{F} \in \mathrm{BMO}\left(B^{+}_{R_{0}};\varLambda^{k+1}\right)$ as well and we have the estimate
		\begin{align*}
			\left[ \tilde{F} \right]_{\mathrm{BMO}\left(B^{+}_{R_{0}};\varLambda^{k+1}\right)} &\leq C\left[ F \right]_{\mathrm{BMO}\left(\Omega;\varLambda^{k+1}\right)}.  
		\end{align*}
	\end{lemma}
	\begin{proof}
		Since $\partial\Omega$ is $C^{2},$ for any $x_{0} \in \partial\Omega,$ there exists a neighborhood $U_{1}$ of $x_0$ in $\mathbb{R}^n$ and a positive number $0 < R_{1} <1$ such that there exists an admissible boundary coordinate system $\Phi\in \operatorname{Diff}^{2}(\overline{B_{R_{1}}};\overline{U_{1}})$, such that 
		\begin{align*}
			\Phi(0)=x_0, \quad D\Phi(0) \in \mathbb{SO}\left(n\right), \quad\Phi(B^{+}_{R_{1}})=\Omega\cap U_{1}, \quad  \Phi(\Gamma_{R_{1}})=\partial\Omega\cap U_{1}. 
		\end{align*}
		This is well-known. See  Lemma B.7 in \cite{csatothesis} for a detailed proof.  
		Now pick $0 < R_{0} < R_{1}$ and define $U := \Phi \left( B_{R_{0}}\right). $ Now choose an open set $U \subset \tilde{U} \subset U_{1}$ such that $ \partial\Omega \cap U \subset  \partial\Omega \cap \tilde{U}$ and $\tilde{\Omega}:= \Omega \cap \tilde{U}$ has $C^{2}$ boundary. Now, by Proposition \ref{weak formulation change linear lemma}, \eqref{omegaequation} is equivalent to 
		\begin{align*}
			\int_{\Omega}\langle A(x)d\omega;d\phi\rangle =\int_{\Omega}\langle F;d\phi\rangle \quad \text{ for all } \phi\in W^{1,2}_{0}\left(\Omega;\varLambda^k\right).
		\end{align*} This in particular implies 
		\begin{align*}
			\int_{\tilde{\Omega}}\langle A(x)d\omega;d\phi\rangle =\int_{\tilde{\Omega}}\langle F;d\phi\rangle \quad \text{ for all } \phi\in C^{\infty}_{c}\left(\tilde{\Omega};\varLambda^k\right).
		\end{align*}
		As $\tilde{\Omega}$ is open, bounded and has $C^{2}$ boundary, applying Proposition \ref{weak formulation change linear lemma} again, we deduce 
		\begin{align*}
			\int_{\tilde{\Omega}}\langle A(x)d\omega;d\phi\rangle =\int_{\tilde{\Omega}}\langle F;d\phi\rangle \quad \text{ for all } \phi\in W^{1,2}_{T}\left(\tilde{\Omega};\varLambda^k\right).
		\end{align*}
		But since $d^{\ast}\omega = 0$ in $\Omega,$ this can also be written as   
		\begin{align*}
			\int_{\tilde{\Omega}}\langle A(x)d\omega;d\phi\rangle + \int_{\tilde{\Omega}}\langle d^{\ast}\omega;d^{\ast}\phi\rangle =\int_{\tilde{\Omega}}\langle F;d\phi\rangle \quad \text{ for all } \phi\in W^{1,2}_{T}\left(\tilde{\Omega};\varLambda^k\right).
		\end{align*}
		Note that $\nu\wedge \omega = 0$ on $\partial\Omega$ and thus, in particular, on $\partial\Omega \cap U.$ Hence, we have 
		\begin{align*}
			e_{n}\wedge u = e_{n}\wedge \Phi^{\ast}\left( \omega \right) = \Phi^{\ast}\left( \nu \right) \wedge \Phi^{\ast}\left( \omega \right) = \Phi^{\ast}\left( \nu \wedge \omega \right) = 0\qquad \text{ on } \Gamma_{R_{0}}.
		\end{align*}
		Now,  $\phi := \left( \Phi^{-1}\right)^{\ast}\left( \psi\right) \in W^{1,2}_{T}\left(\tilde{\Omega};\varLambda^k\right)$ for any $\psi\in W^{1,2}_{T, \text{flat}}(B^{+}_{R_{0}};\varLambda^{k}).$ Thus,  
		\begin{multline*}
			\int_{\Phi \left( B^{+}_{R_{0}}\right)}\langle A(x)d\left((\Phi^{-1})^{\ast}u\right);d\left((\Phi^{-1})^{\ast}\psi\right)\rangle \\+\int_{\Phi \left( B^{+}_{R_{0}}\right)}\langle d^{\ast}\left((\Phi^{-1})^{\ast}u\right);d^{\ast}\left((\Phi^{-1})^{\ast}\psi\right)\rangle 
			=\int_{\Phi \left( B^{+}_{R_{0}}\right)}\langle F;d\left((\Phi^{-1})^{\ast}\psi\right)\rangle,
		\end{multline*}
		for all $\psi \in W^{1,2}_{T, \text{flat}}(B^{+}_{R_{0}};\varLambda^{k}).$
		Set 
		\begin{align}\label{formula for tilde A tilde F flattening}
			\left\lbrace \begin{aligned}
				\Tilde{A}\left( y\right)&=\left\lvert \operatorname{det} D\Phi(y) \right\rvert A \left( \Phi(y)\right) , \\
				\Tilde{F}\left( y\right)&=\left\lvert \operatorname{det} D\Phi(y) \right\rvert F \left( \Phi(y)\right),  
			\end{aligned} \right. \qquad \text{ for a.e. } y \in B_{R_{0}}^{+}. 
		\end{align}	
		Since $d$ commutes with pullback, by change of variables formula, we deduce 
		\begin{align*}
			\int_{\Phi \left( B^{+}_{R_{0}}\right)}&\left\langle
			A(x)d\left((\Phi^{-1})^{\ast}u\right);d\left((\Phi^{-1})^{\ast}\psi\right)\right\rangle \\&=\int_{\Phi \left( B^{+}_{R_{0}}\right)}\left\langle A(x)(\Phi^{-1})^{\ast}du;(\Phi^{-1})^{\ast}d\psi\right\rangle\\
			&=\int_{B_{R_{0}}^{+}}\left\langle A\left(\Phi(y)\right)du(y);d\psi(y)\right\rangle\left\lvert \operatorname{det} D\Phi(y) \right\rvert	\\ &=\int_{B_{R_{0}}^{+}}\left\langle \Tilde{A}\left( y\right) du(y);d\psi(y)\right\rangle. 
		\end{align*}
		Similarly, we have 
		\begin{align*}
			\int_{\Phi \left( B^{+}_{R_{0}}\right)}\langle F;d\left((\Phi^{-1})^{\ast}\psi\right)\rangle=\int_{B_{R_{0}}^{+}}\langle\Tilde{F};d\psi\rangle.
		\end{align*}
		Now, by computing out the derivatives, using change of variables formula to transfer all integrals to $B_{R_{0}}^{+}$ and grouping the terms, we can write
		\begin{multline*}
			\int_{\Phi \left( B^{+}_{R_{0}}\right)}\langle d^{\ast}\left((\Phi^{-1})^{\ast}u\right);d^{\ast}\left((\Phi^{-1})^{\ast}\psi\right)\rangle -\int_{B_{R_{0}}^{+}}\langle d^{\ast}u;d^{\ast}\psi\rangle \\= \int_{B_{R_{0}}^{+}} \langle \mathrm{P}u +  \mathrm{R}\nabla u ; \psi \rangle +  \int_{B_{R_{0}}^{+}}\langle \mathrm{Q} u ;\nabla\psi\rangle+\int_{B_{R_{0}}^{+}}\langle \mathrm{S}\nabla u;\nabla \psi\rangle.
		\end{multline*}
		Since these coefficients contains up to second order derivatives of $\Phi,$ with $\mathrm{S}$ containing only up to first order derivatives of $\Phi,$ the claimed regularity properties follow. Note that since $\mathrm{S}$ is $C^{1},$  to show the estimate 
		$$ \left\lVert \mathrm{S} \right\rVert_{L^{\infty}\left(B_{r}^{+}\right)} \leq Cr \qquad \text{ for any } 0 < r< R_{0}, $$ it is enough to show that $\mathrm{S}\left(0\right) =0.$ See Lemma 4.17 in \cite{csatothesis} for a proof. This follows from the fact for any $M \in \mathbb{SO}\left(n\right),$ if we set 
		$\Psi(y)=My$ in $B_R, $	then 
		\begin{align*}
			\langle d^{\ast}\left((\Psi^{-1})^{\ast}u\right);d^{\ast}\left((\Psi^{-1})^{\ast}\psi\right)\rangle - 	\left\langle d^{\ast}u\circ\Psi^{-1};d^{\ast}\psi\circ\Psi^{-1}\right\rangle = 0. 
		\end{align*}
		The estimates for $\tilde{A}$ and $\tilde{F}$ follow from \eqref{formula for tilde A tilde F flattening}.
	\end{proof}
	\subsection{Global estimates}
	\begin{theorem}\label{Main theorem BMO linear A only}
		Let $\Omega \subset \mathbb{R}^{n}$ be  open, bounded with $\partial \Omega \in C^{2}.$   Let $A \in L^{\infty}\cap \mathrm{V}\mathscr{L}_{\left( 1+ \left\lvert \log r \right\rvert \right)^{-1}}^{2,n}\left( \Omega; \mathcal{L}\left( \varLambda^{k+1}; \varLambda^{k+1}\right)\right)$ be uniformly Legendre-elliptic with ellipticity constant $\gamma >0$ and let $F \in \mathrm{BMO}\left( \Omega; \varLambda^{k+1}\right)$.
		If $u \in W^{1,2}_{d^{\ast}, T}\left( \Omega ; \varLambda^{k}\right)$ is a weak solution for the following 
		\begin{align}\label{Main}
			\left\lbrace \begin{aligned}
				d^{\ast} \left( A\left( x \right) du \right)  &= d^{\ast} F &&\text{ in } \Omega, \\
				d^{\ast} u &= 0 &&\text{  in } \Omega, \\
				\nu\wedge u &= 0 &&\text{ on } \partial\Omega,
			\end{aligned}	\right. 
		\end{align}
		then $\nabla u \in \mathrm{BMO} \left( \Omega,; \varLambda^{k + 1}\right)$ and there exists a constant 
		\begin{align*}
			C = C\left( n, k, N, \gamma, \Omega, \Theta_{\frac{1}{1+ \left\lvert \log r \right\rvert}}^{A}, \|A\|_{L^{\infty}(\Omega)}, \left[ A \right]_{\mathrm{BMO}^{2}_{\frac{1}{1+ \left\lvert \log r \right\rvert}}\left(\Omega\right)} \right) >0
		\end{align*} such that we have the estimates
		\begin{align}\label{du bmo estimate A only}
			\left[ du\right]_{\mathrm{BMO}\left(\Omega\right)} &\leq   C \left(  \left[ F\right]_{\mathrm{BMO}\left( \Omega\right)} + \left\lVert F \right\rVert_{L^{2}\left(\Omega\right)}\right) \intertext{ and } 
			\left[ \nabla u\right]_{\mathrm{BMO}\left(\Omega\right)} &\leq   C \left(  \left[ F\right]_{\mathrm{BMO}\left( \Omega\right)} + \left\lVert F \right\rVert_{L^{2}\left(\Omega\right)} + \left\lVert u \right\rVert_{L^{2}\left(\Omega\right)}\right). \label{grad u bmo estimate A only}
		\end{align}
	\end{theorem}
	\begin{remark}
		By standard contradiction-compactness arguments, the $L^{2}$ norm of $u$ on the RHS of \eqref{grad u bmo estimate A only} can be dropped if uniqueness holds for \eqref{Main}. Proposition \ref{poincare gaffney} and ellipticity of $A$ implies that uniqueness holds if and only if   $\mathcal{H}^{k}_{T}= \left\lbrace 0 \right\rbrace. $
	\end{remark}
	\begin{proof} We can assume $\mathcal{H}^{k}_{T}\left(\Omega; \varLambda^{k}\right)= \left\lbrace 0 \right\rbrace.$ Indeed, if not, then for any $u \in W^{1,2}_{d^{\ast}, T}\left( \Omega ; \varLambda^{k}\right)$ we have the orthogonal decomposition 
		\begin{align*}
			u =v + h, \qquad \text{ and } \left\lVert u\right\rVert_{L^{2}}^{2} = \left\lVert v\right\rVert_{L^{2}}^{2} + \left\lVert h\right\rVert_{L^{2}}^{2},
		\end{align*}
		where $v \in W^{1,2}_{d^{\ast}, T}\cap \left( \mathcal{H}^{k}_{T}\right)^{\perp}$ and $h \in \mathcal{H}^{k}_{T}.$ Since $h$ is a harmonic field, $dh = 0$ and  we have the estimate $\left\lVert \nabla h \right\rVert_{L^{\infty}\left( \Omega;\varLambda^{k}\right)} \leq C \left\lVert h\right\rVert_{L^{2}\left(\Omega ;\varLambda^{k}\right)}.$ Thus if the estimates \eqref{du bmo estimate A only} and \eqref{grad u bmo estimate A only} for $v,$ then they also hold for $u.$ By Lemma \ref{approximation lemma}, we can also assume that $A$ is smooth. The regularity results ( see e.g \cite{Sil_linearregularity} ) for smooth coefficients implies $\nabla u$ is $BMO$ in $\Omega.$ So we show the estimates \eqref{du bmo estimate A only} and \eqref{grad u bmo estimate A only}. \smallskip 
		
		By Lemma \ref{boundary estimate lemma}, for every $x \in \partial\Omega,$ there exists $0 < R_{x} < 1,$ a neighborhood $W_{x}$ of $x$ in $\mathbb{R}^{n}$ and diffeomorphisms $\Phi_{x}:B_{R_{x}}\rightarrow W_{x}  $ such that the conclusions of Lemma \ref{boundary estimate lemma} holds. By compactness of $\partial\Omega,$ we can find finitely many points $x_{1}, \ldots, x_{l} \in \partial\Omega$ such that $	\partial\Omega \subset \bigcup\limits_{i=1}^{l} \Phi_{x_{i}}\left(B^{+}_{R_{x_{i}}/2}\right).$
		Choose $\Omega_{0}$ such that $\Omega_{0} \subset \subset \Omega,$  $d_{0}:= \operatorname{dist}\left(\Omega_{0}, \partial \Omega \right) < 1$ and 
		\begin{align*}
			\Omega \subset \bigcup\limits_{i=0}^{l} \Omega_{i}, \quad \text{ where } \Omega_{i} = \Phi_{x_{i}}\left(B^{+}_{R_{x_{i}}/2}\right) \text{ for } 1 \leq i \leq l.
		\end{align*}
		Clearly, we have 
		\begin{align}\label{gradu bmo estimate sum of splitting}
			\left[ \nabla u\right]^{2}_{\mathrm{BMO}\left(\Omega\right)} \leq \left[ \nabla u\right]_{\mathrm{BMO}\left(\Omega_{0}\right)}^{2} + \sum\limits_{i=1}^{l} 	\left[ \nabla u \right]_{\mathrm{BMO}\left( \Phi_{x_{i}} \left( B^{+}_{R_{x_{i}}/2}\right)\right)}^{2}.
		\end{align}
		Fix $0 < \sigma < d_{0}/2.$ Then for any $y \in \Omega_{0},$ we have $B_{\sigma}\left( y\right) \subset \subset \Omega.$ By Theorem 9 in \cite{Sil_linearregularity}, we can find $\alpha \in W^{1,2}_{d^{\ast}, T}\left( B_{\sigma}\left( y\right); \varLambda^{k}\right)$ which is a weak solution of 
		\begin{align}\label{picking up the rhs interior}
			\int_{B_{\sigma}\left( y\right)}\left\langle  \left( A\right)_{B_{\sigma}\left( y\right)} d\alpha; d\psi \right\rangle  = 	\int_{B_{\sigma}\left( y\right)}\left\langle F; d\psi\right\rangle  - 	\int_{B_{\sigma}\left( y\right)}\left\langle G_{3}; d\psi \right\rangle 
		\end{align}
		for all $\psi \in W^{1,2}_{d^{\ast}, T}\left( B_{\sigma}\left( y\right); \varLambda^{k}\right), $ where $	G_{3}:= \left[ A\left(x\right) - \left( A\right)_{B_{\sigma}\left( y\right)}\right]du. $ Plugging $\psi = \alpha$ in \eqref{picking up the rhs interior} and Young's inequality with $\varepsilon>0,$ we have 
		\begin{align*}
			\int_{B_{\sigma}\left( y\right)} \left\lvert d\alpha\right\rvert^{2} &\leq \frac{C}{\gamma}	\int_{B_{\sigma}\left( y\right)} \left\langle \left( A\right)_{B_{\sigma}\left( y\right)}d\alpha;  d\alpha\right\rangle  \\&=\frac{C}{\gamma}\left(  \int_{B_{\sigma}\left( y\right)} \left\langle F;  d\alpha\right\rangle - \int_{B_{\sigma}\left( y\right)} \left\langle G_{3};  d\alpha\right\rangle \right)  \\&= \frac{C}{\gamma} \left( \int_{B_{\sigma}\left( y\right)} \left\langle F - \left(F\right)_{B_{\sigma}\left( y\right)};  d\alpha\right\rangle  - \int_{B_{\sigma}\left( y\right)} \left\langle G_{3};  d\alpha\right\rangle\right) \\
			&\leq \frac{C}{\gamma}\left( \left\lvert \int_{B_{\sigma}\left( y\right)} \left\langle F - \left(F\right)_{B_{\sigma}\left( y\right)};  d\alpha\right\rangle\right\rvert + \left\lvert \int_{B_{\sigma}\left( y\right)} \left\langle G_{3};  d\alpha\right\rangle\right\rvert \right) \\
			&\leq \varepsilon 	\int_{B_{\sigma}\left( y\right)} \left\lvert d\alpha\right\rvert^{2} + C_{\varepsilon}\left( 	\int_{B_{\sigma}\left( y\right)} \left\lvert F -  \left(F\right)_{B_{\sigma}\left( y\right)}\right\rvert^{2} + \int_{B_{\sigma}\left( y\right)} \left\lvert G_{3}\right\rvert^{2}\right) .
		\end{align*}
		Thus, choosing $\varepsilon>0$ small enough and Proposition \ref{Gaffney with no L2 term}, we deduce 
		\begin{align*}
			\int_{B_{\sigma}\left( y\right)} \left\lvert \nabla \alpha\right\rvert^{2} \leq C \int_{B_{\sigma}\left( y\right)} \left\lvert d\alpha\right\rvert^{2} \leq C \left( \int_{B_{\sigma}\left( y\right)} \left\lvert F - \left(F\right)_{B_{\sigma}\left( x\right)}\right\rvert^{2} +  \int_{B_{\sigma}\left( y\right)} \left\lvert G_{3}\right\rvert^{2}\right).
		\end{align*}
		Now the last term is estimated exactly as was done in Lemma \ref{boundary estimate lemma} and we deduce 
		\begin{multline}\label{dalpha esti interior interior}
			\int_{B_{\sigma}\left( y\right)} \left\lvert \nabla \alpha\right\rvert^{2} \leq C\int_{B_{\sigma}\left( y\right)} \left\lvert F - \left(F\right)_{B_{\sigma}\left( y\right)}\right\rvert^{2} + c\sigma^{n}\left[ \Theta^{A}\left( \sigma\right)\right]^{2}\left[ \nabla u\right]_{\mathrm{BMO}\left(\Omega\right)}^{2}  \\ + c\left[ A\right]_{\mathscr{L}^{2,n}_{\frac{1}{1+ \left\lvert \log r \right\rvert}}\left(\Omega\right)}^{2}\left\lVert \nabla u \right\rVert_{L^{2}\left( \Omega \right)}^{2}. 
		\end{multline}
		Note $\beta := u -\alpha$ satisfies the homogeneous constant coefficient system 
		\begin{align}\label{homogeneous interior}
			\int_{B_{\sigma}\left( y\right)}\left\langle  \left( A\right)_{B_{\sigma}\left( y\right)} d\beta; d\psi \right\rangle +  	\int_{B_{\sigma}\left( y\right)}\left\langle d^{\ast}\beta; d^{\ast}\psi \right\rangle= 0 
		\end{align}
		for all $\psi \in W^{1,2}_{d^{\ast}, T}\left( B_{\sigma}\left( y\right); \varLambda^{k}\right),$ as $d^{\ast}\beta = 0.$
		Hence $\beta$ satisfies the decay estimates ( see Theorem 3 in \cite{Sil_linearregularity} ). Thus, by standard arguments, for any $0 < \rho < \sigma ,$ we have 
		\begin{align*}
			\int_{B_{\rho}\left( y \right)} &\left\lvert \nabla u - \left( \nabla u \right)_{B_{\rho}\left( y \right)}\right\rvert^{2} \\
			&\leq c \left(\frac{\rho}{\sigma}\right)^{n+2} 	\int_{B_{\sigma}\left( y \right)} \left\lvert \nabla u - \left( \nabla u \right)_{B_{\sigma}\left( y \right)}\right\rvert^{2} + c\int_{B_{\sigma}\left( y \right)} \left\lvert \nabla \alpha \right\rvert^{2}. 
		\end{align*}
		By \eqref{dalpha esti interior interior} and the iteration lemma, for any $y \in \Omega_{0},$ this implies 
		\begin{align*}
			\fint_{B_{\rho}\left( y \right)} &\left\lvert \nabla u - \left( \nabla u \right)_{B_{\rho}\left( y \right)}\right\rvert^{2} \\
			&\begin{multlined}[b]
				\leq c\fint_{B_{\sigma}\left( y\right)} \left\lvert F - \left(F\right)_{B_{\sigma}\left( y\right)}\right\rvert^{2} + c\left[ \Theta^{A}\left( \sigma\right)\right]^{2}\left[ \nabla u\right]_{\mathrm{BMO}\left(\Omega\right)}^{2}  \\ + c\sigma^{-n}\left( 1 + \left[ A\right]_{\mathscr{L}^{2,n}_{\frac{1}{1+ \left\lvert \log r \right\rvert}}\left(\Omega\right)}^{2}\right) \left\lVert \nabla u \right\rVert_{L^{2}\left( \Omega \right)}^{2}, \quad \text{ for any }0 < \rho < \sigma .
			\end{multlined} 
		\end{align*} 
		In view of the obvious estimate for $\rho \geq \sigma,$ taking supremum, we arrive at
		\begin{align}\label{grad u bmo estimate Omega0}
			\left[ \nabla u \right]_{\mathrm{BMO}\left( \Omega_{0}\right)}^{2} \leq C_{0} \left[ \Theta^{A}\left( \sigma\right)\right]^{2}\left[ \nabla u \right]_{\mathrm{BMO}\left(\Omega\right)}^{2} + \sigma^{-n}\kappa_{0} , 
		\end{align}
		for every $0 < \sigma < d_{0}/2,$	where 
		\begin{align*}
			\kappa_{0} := \left(1 +	\left[ A\right]_{\mathscr{L}^{2,n}_{\frac{1}{1+ \left\lvert \log r \right\rvert}}\left(\Omega\right)}^{2}\right)\left\lVert \nabla u  \right\rVert_{L^{2}\left(\Omega\right)}^{2} + \left[ F\right]_{\mathrm{BMO}\left(\Omega\right)}^{2}. 
		\end{align*}
		
		By Lemma \ref{boundary estimate lemma}, for each $1 \leq i \leq l,$ there exists constants $C_{i}, c_{1}^{i}>0$ such that 
		\begin{align}\label{grad u bmo boundary estimate i from 1 to l}
			\left[ \nabla u \right]&_{\mathrm{BMO}\left( \Phi_{x_{i}} \left( B^{+}_{R_{x_{i}}/2}\right)\right)}^{2} \notag \\&\leq C_{i} \left\lbrace \left(   \left[ \Theta^{A}\left( 2c^{i}_{1}\sigma\right)\right]^{2} + \left[ \theta \left( \sigma\right)\right]^{2}\left( \left\lVert A \right\rVert_{L^{\infty}}^{2} + 1 \right)\right) \left[ \nabla u \right]_{\mathrm{BMO}\left(\Omega\right)}^{2} + \sigma^{-n}\kappa \right\rbrace , 
		\end{align}
		for every $0 < \sigma < R_{x_{i}}/8,$ where 
		\begin{multline}\label{definition of kappa}
			\kappa := \left(\left\lVert A \right\rVert_{L^{\infty}\left( \Omega\right)}^{2} + 	\left[ A\right]_{\mathscr{L}^{2,n}_{\frac{1}{1+ \left\lvert \log r \right\rvert}}\left(\Omega\right)}^{2} + 1 \right)\left\lVert \nabla u  \right\rVert_{L^{2}\left(\Omega\right)}^{2} + \left\lVert u  \right\rVert_{L^{2}\left(\Omega\right)}^{2} \\ + \left\lVert F \right\rVert_{L^{2}\left(\Omega\right)}^{2} + \left[ F\right]_{\mathrm{BMO}\left(\Omega\right)}^{2}. 
		\end{multline}
		Take $c^{0}_{1}=1$ and set 
		\begin{align*}
			\sigma_{1} := \min \left\lbrace \frac{d_{0}}{2}, \min\limits_{1\leq i\leq l}\left\lbrace \frac{R_{x_{i}}}{8}\right\rbrace \right\rbrace ,  \tilde{C}:= \max\limits_{0 \leq i \leq l} \left\lbrace C_{i}\right\rbrace \text{ and } \tilde{c}_{1}:= \max\limits_{0 \leq i \leq l} \left\lbrace c^{i}_{1}\right\rbrace. 
		\end{align*}
		Now choose $0 < \sigma_{0} < \sigma_{1}$ small enough such that
		\begin{align}
			\tilde{C} \left(   \left[ \Theta^{A}\left( 2\tilde{c}_{1}\sigma_{0}\right)\right]^{2} + \left[ \theta \left( \sigma_{0}\right)\right]^{2}\left( \left\lVert A \right\rVert_{L^{\infty}}^{2} + 1 \right)\right) &\leq \frac{1}{2\left(l +1\right)} . 
		\end{align}	
		In view of \eqref{gradu bmo estimate sum of splitting}, \eqref{grad u bmo estimate Omega0} and \eqref{grad u bmo boundary estimate i from 1 to l}, this implies the estimate 
		\begin{align*}
			\left[ \nabla u\right]^{2}_{\mathrm{BMO}\left(\Omega\right)} &\leq \left(\sum\limits_{j=0}^{l} \frac{1}{2\left(l +1\right)}\right)\left[ \nabla u \right]_{\mathrm{BMO}\left(\Omega\right)}^{2} + \tilde{C}\sigma_{0}^{-n}\kappa \\
			&\leq \frac{1}{2}\left[ \nabla u \right]_{\mathrm{BMO}\left(\Omega\right)}^{2} + \tilde{C}\sigma_{0}^{-n}\kappa, 
		\end{align*}
		where $\kappa$ is as in \eqref{definition of kappa} (note $\kappa_{0} \leq \kappa$). But this implies the estimate
		\begin{align}\label{prefinal estimate}
			\left[ \nabla u\right]^{2}_{\mathrm{BMO}\left(\Omega\right)} \leq C \left( \left\lVert u  \right\rVert_{W^{1,2}\left(\Omega\right)}^{2}  + \left\lVert F \right\rVert_{L^{2}\left(\Omega\right)}^{2} + \left[ F\right]_{\mathrm{BMO}\left(\Omega\right)}^{2}\right), 
		\end{align}
		where the constant $C >0$ depends only on 
		$$ n, k, N, \gamma, \Omega, \Theta^{A}, \left\lVert A\right\rVert_{L^{\infty}(\Omega)} \text{ and } \left[ A \right]_{\mathscr{L}^{2,n}_{\frac{1}{1+ \left\lvert \log r \right\rvert}}\left(\Omega\right)}. $$ 
		But since  $u \in W^{1,2}_{d^{\ast}, T}$ solves \eqref{Main}, using Young's inequality with $\varepsilon>0$ small enough and Proposition \ref{poincare gaffney},  we have
		\begin{align*}
			\left\lVert u  \right\rVert_{W^{1,2}\left(\Omega\right)}^{2} \leq C \left\lVert du  \right\rVert_{L^{2}\left(\Omega\right)}^{2} \leq C\left\lVert F \right\rVert_{L^{2}\left( \Omega \right)}^{2}.
		\end{align*}
		Combining this with \eqref{prefinal estimate}, we have \eqref{grad u bmo estimate A only} and consequently, \eqref{du bmo estimate A only}. 
	\end{proof}
	\subsection{Uniqueness}
	Now we discuss the uniqueness of solutions for our systems. 
	\begin{lemma}\label{uniqueness lemma}
		Let $\Omega \subset \mathbb{R}^{n}$ be  open, bounded with $\partial \Omega \in C^{2}.$  Let $A \in L^{\infty}\cap \mathrm{V}\mathscr{L}_{\left( 1+ \left\lvert \log r \right\rvert \right)^{-1}}^{2,n}\left( \Omega; \mathcal{L}\left( \varLambda^{k+1}; \varLambda^{k+1}\right)\right),$ $B \in L^{\infty}\cap \mathrm{V}\mathscr{L}_{\left( 1+ \left\lvert \log r \right\rvert \right)^{-1}}^{2,n}\left( \Omega; \mathcal{L}\left( \varLambda^{k}; \varLambda^{k}\right)\right)$ be uniformly Legendre-elliptic with ellipticity constants $\gamma_{A}, \gamma_{B} >0$, respectively. Let $u \in W^{d,2}_{T}\left( \Omega ; \varLambda^{k}\right)$ be a weak solution to the following 
		\begin{align}\label{Main uniqueness}
			\left\lbrace \begin{aligned}
				d^{\ast} \left( A\left( x \right) du \right)  &= 0 &&\text{ in } \Omega, \\
				d^{\ast} \left( B\left(x\right) u\right) &=0 &&\text{  in } \Omega, \\
				\nu\wedge u &= 0 &&\text{ on } \partial\Omega,
			\end{aligned}	\right. 
		\end{align}
		Then $u=0$ if and only if $\mathcal{H}^{k}_{T}\left(  \Omega;\varLambda^{k}\right) = \left\lbrace 0 \right\rbrace.$
	\end{lemma}
	\begin{proof}
		Suppose there exists a non-zero $h \in \mathcal{H}^{k}_{T}\left(  \Omega;\varLambda^{k}\right).$ Since $h$ is smooth, we can use Theorem \ref{Main theorem BMO linear A only} to find $\alpha \in W^{1,2}_{d^{\ast}, T}\left( \Omega ; \varLambda^{k-1}\right),$ a weak solution to
		\begin{align*}
			\left\lbrace \begin{aligned}
				d^{\ast} \left( B\left( x \right) d\alpha \right)  &= d^{\ast} \left( B\left( x \right) h \right) &&\text{ in } \Omega, \\
				d^{\ast} \alpha &= 0 &&\text{  in } \Omega, \\
				\nu\wedge \alpha &= 0 &&\text{ on } \partial\Omega,
			\end{aligned}	\right. 
		\end{align*}
		such that $\nabla \alpha \in \mathrm{BMO}.$ Then setting $u = d\alpha -h,$ it is easy to see that $u \in W^{d,2}_{T}\left( \Omega ; \varLambda^{k}\right)$ is a solution to \eqref{Main uniqueness}. But as $u=0$ would imply $h=d\alpha,$ which is impossible as no nontrivial harmonic field can be exact, $u$ is a nontrivial solution to \eqref{Main uniqueness}. For the reverse implication, ellipticity of $A$ implies we must have $du = 0.$ Since $\mathcal{H}^{k}_{T}\left(  \Omega;\varLambda^{k}\right) = \left\lbrace 0 \right\rbrace,$ by Hodge decomposition, there exists $\beta \in W^{1,2}_{d^{\ast}, T}\left( \Omega ; \varLambda^{k-1}\right)$ such that $u= d\beta.$ But now ellipticity of $B$ and Proposition \ref{poincare gaffney} implies $\beta = 0.$ This completes the proof. 
	\end{proof}
	\section{Main results}\label{main results as corollaries}
	\begin{theorem}\label{Main theorem BMO linear A u0}
		Let $\Omega \subset \mathbb{R}^{n}$ be  open, bounded with $\partial \Omega \in C^{2}.$   Let $A \in L^{\infty}\cap \mathrm{V}\mathscr{L}_{\left( 1+ \left\lvert \log r \right\rvert \right)^{-1}}^{2,n}\left( \Omega; \mathcal{L}\left( \varLambda^{k+1}; \varLambda^{k+1}\right)\right)$ be uniformly Legendre-elliptic with ellipticity constant $\gamma >0$ and let $F \in \mathrm{BMO}\left( \Omega; \varLambda^{k+1}\right)$. Let $u_{0}\in W^{1,2}\left( \Omega ; \varLambda^{k}\right)$ be such that $\nabla u_{0} \in \mathrm{BMO}\left( \Omega; \varLambda^{k}\otimes\mathbb{R}^{n}\right).$
		Then there exists $u \in W^{1,2}\left( \Omega ; \varLambda^{k}\right),$ which is a weak solution for the following 
		\begin{align}\label{Main A}
			\left\lbrace \begin{aligned}
				d^{\ast} \left( A\left( x \right) du \right)  &= d^{\ast} F &&\text{ in } \Omega, \\
				d^{\ast} u &= 0 &&\text{  in } \Omega, \\
				\nu\wedge u &= \nu\wedge u_{0} &&\text{ on } \partial\Omega,
			\end{aligned}	\right. 
		\end{align}
		then $\nabla u \in \mathrm{BMO} \left( \Omega; \varLambda^{k}\otimes \mathbb{R}^{n}\right)$ and there exists a constant 
		\begin{align*}
			C = C\left( n, k, N, \gamma, \Omega, \Theta^{A}, \|A\|_{L^{\infty}}, \left[ A \right]_{\mathscr{L}^{2,n}_{\frac{1}{1+ \left\lvert \log r \right\rvert}}} \right) >0
		\end{align*} such that we have the estimates
		\begin{align}\label{du bmo estimate A u0}
			\left[ du\right]_{\mathrm{BMO}\left(\Omega\right)} &\leq   C \left(  \left[ F\right]_{\mathrm{BMO}\left( \Omega\right)} + \left\lVert F \right\rVert_{L^{2}\left(\Omega\right)} + \left[ du_{0}\right]_{\mathrm{BMO}\left( \Omega\right)} + \left\lVert du_{0} \right\rVert_{L^{2}\left(\Omega\right)}\right)
		\end{align}
		and 
		\begin{align}
			\left[ \nabla u\right]_{\mathrm{BMO}\left(\Omega\right)} \leq   C \left(  \left\lVert u \right\rVert_{L^{2}\left(\Omega\right)} + \left[ \left( F, \nabla u_{0} \right)\right]_{\mathrm{BMO}\left( \Omega\right)} + \left\lVert \left( F, \nabla u_{0} \right) \right\rVert_{L^{2}\left(\Omega\right)} \right). \label{grad u bmo estimate A u0}
		\end{align}
		Moreover, the solution is unique if and only if $\mathcal{H}^{k}_{T}\left(  \Omega;\varLambda^{k}\right) = \left\lbrace 0 \right\rbrace.$
	\end{theorem}
	\begin{remark}
		If $\mathcal{H}^{k}_{T} \neq \left\lbrace 0 \right\rbrace,$ then any weak solution $v \in W^{1,2}_{d^{\ast}, T}\left( \Omega ; \varLambda^{k}\right)$ of \eqref{Main A} differs from $u$ by a harmonic field $h \in \mathcal{H}^{k}_{T}$ and thus all weak solutions satisfy the estimates \eqref{du bmo estimate A u0} and \eqref{grad u bmo estimate A u0}. Also, if $\mathcal{H}^{k}_{T} = \left\lbrace 0 \right\rbrace,$ then the $L^{2}$ norm of $u$ can be dropped from the RHS of the estimate \eqref{grad u bmo estimate A u0} in view of uniqueness. 
	\end{remark}
	\begin{proof}
		First we find $\alpha \in W^{1,2}_{d^{\ast}, T}\left( \Omega ; \varLambda^{k}\right) \cap \left( \mathcal{H}^{k}_{T}\right)^{\perp}$, a weak solution to
		\begin{align*}
			\left\lbrace \begin{aligned}
				d^{\ast} \left( A\left( x \right) d\alpha \right)  &= d^{\ast} F - d^{\ast}\left( A \left(x\right) du_{0}\right)&&\text{ in } \Omega, \\
				d^{\ast} \alpha &= 0 &&\text{  in } \Omega, \\
				\nu\wedge \alpha &= 0 &&\text{ on } \partial\Omega.
			\end{aligned}	\right. 
		\end{align*}
		Existence follows from applying Lax-Milgram in the space $W^{1,2}_{d^{\ast}, T}$ ( see \cite{Sil_linearregularity} ). Now, since $\nabla u_{0} \in \mathrm{BMO}\left( \Omega; \varLambda^{k}\otimes\mathbb{R}^{n}\right),$ we have $du_{0} \in \mathrm{BMO}\left( \Omega; \varLambda^{k+1}\right)$ and using Remark \ref{Remark about Mean Oscillation space properties} (ii) componentwise, we deduce $A \left(x\right) du_{0} \in \mathrm{BMO}\left( \Omega; \varLambda^{k+1}\right).$ Thus, Theorem \ref{Main theorem BMO linear A only} implies $\nabla \alpha \in \mathrm{BMO}$ with corresponding estimates. Now we find $\beta \in W^{1,2}_{d^{\ast}, T}\left( \Omega ; \varLambda^{k}\right)\cap \left( \mathcal{H}^{k}_{T}\right)^{\perp},$ a weak solution for the following 
		\begin{align*}
			\left\lbrace \begin{aligned}
				d\beta  &= 0 &&\text{ in } \Omega, \\
				d^{\ast} \beta &= - d^{\ast}u_{0} &&\text{  in } \Omega, \\
				\nu\wedge \beta &= 0 &&\text{ on } \partial\Omega.
			\end{aligned}	\right. 
		\end{align*}
		By standard estimates, (see e.g \cite{Sil_linearregularity} ), we deduce $\nabla \beta \in \mathrm{BMO}$ with corresponding estimates. Finally, we set $u = \alpha + \beta + u_{0}.$
	\end{proof}
	As a consequence and Stampacchia interpolation and duality, we also have the following. 
	\begin{theorem}\label{Main theorem Lp linear A u0}
		Let $\Omega \subset \mathbb{R}^{n}$ be  open, bounded with $\partial \Omega \in C^{2}.$  Let $1 < p < \infty.$ Let $A \in L^{\infty}\cap \mathrm{V}\mathscr{L}_{\left( 1+ \left\lvert \log r \right\rvert \right)^{-1}}^{2,n}\left( \Omega; \mathcal{L}\left( \varLambda^{k+1}; \varLambda^{k+1}\right)\right)$ be uniformly Legendre-elliptic with ellipticity constant $\gamma >0$ and let $F \in L^{p}\left( \Omega; \varLambda^{k+1}\right)$. Let $u_{0}\in W^{1,p}\left( \Omega ; \varLambda^{k}\right)$ Then there exists $u \in W^{1,p}\left( \Omega ; \varLambda^{k}\right),$ which satisfies  
		\begin{align*}
			\left\lbrace \begin{aligned}
				d^{\ast} \left( A\left( x \right) du \right)  &= d^{\ast} F &&\text{ in } \Omega, \\
				d^{\ast} u &= 0 &&\text{  in } \Omega, \\
				\nu\wedge u &= \nu\wedge u_{0} &&\text{ on } \partial\Omega,
			\end{aligned}	\right. 
		\end{align*}
		in the weak sense and there exists a constant 
		\begin{align*}
			C = C\left( n, k, N, \gamma, \Omega, p, \Theta^{A}, \|A\|_{L^{\infty}}, \left[ A \right]_{\mathscr{L}^{2,n}_{\frac{1}{1+ \left\lvert \log r \right\rvert}}} \right) >0
		\end{align*} such that we have the estimates
		\begin{align*}
			\left\lVert du\right\rVert_{L^{p}\left(\Omega\right)} &\leq   C \left( \left\lVert F \right\rVert_{L^{p}\left(\Omega\right)} +  \left\lVert du_{0} \right\rVert_{L^{p}\left(\Omega\right)}\right)
		\end{align*}
		and 
		\begin{align*}
			\left\lVert \nabla u\right\rVert_{L^{p}\left(\Omega\right)}  \leq   C \left(  \left\lVert u \right\rVert_{L^{p}\left(\Omega\right)} + \left\lVert  F \right\rVert_{L^{p}\left(\Omega\right)} + \left\lVert  u_{0} \right\rVert_{W^{1,p}\left(\Omega\right)}\right). 
		\end{align*}
		Moreover, the solution is unique if and only if $\mathcal{H}^{k}_{T}\left(  \Omega;\varLambda^{k}\right) = \left\lbrace 0 \right\rbrace.$
	\end{theorem}
	\begin{proof}
		We can obviously assume $u_{0}=0.$ Now if $p\geq 2,$ consider the `solution operator' that maps 
		$F \mapsto du$, where $u$ is the unique weak solution in $W^{1,2}_{d^{\ast}, T}\left( \Omega ; \varLambda^{k}\right) \cap \left( \mathcal{H}^{k}_{T}\right)^{\perp}.$ By standard $L^{2}$ estimates, this operator is bounded from $L^{2}$ to $L^{2}.$ 
		Theorem \ref{Main theorem BMO linear A only} proves the operator is also bounded from $\mathrm{BMO}$ to $\mathrm{BMO},$ where the operator norm in both cases can be bounded by a constant 
		\begin{align*}
			C = C\left( n, k, N, \gamma, \Omega, \Theta_{\frac{1}{1+ \left\lvert \log r \right\rvert}}^{A}, \|A\|_{L^{\infty}(\Omega)}, \left[ A \right]_{\mathrm{BMO}^{2}_{\frac{1}{1+ \left\lvert \log r \right\rvert}}\left(\Omega\right)} \right) >0. 
		\end{align*} Thus, Stampacchia's interpolation theorem ( see \cite{StampacchiaInterpolation} ) can be used in the standard way ( see Chapter 6 and 7 of \cite{giaquinta-martinazzi-regularity}, also Section 3 in \cite{Acquistapace_linear_elliptic_discontinuous}  ) to prove this operator is bounded from $L^{p}$ to $L^{p}$ for any $2 \leq p < \infty.$ Now the apriori $L^{p}$ estimates in the case $1 < p < 2$ follows by standard duality arguments when $F \in L^{2}\cap L^{p}.$ The existence for the case $1 < p < 2$ now follows from these estimates by approximating $F \in L^{p}$ by a sequence $\left\lbrace F_{s} \right\rbrace_{s \in \mathbb{N}} \subset L^{2}\cap L^{p}$ and passing to the limit. 	\end{proof}

	\begin{theorem}\label{Main theorem BMO linear AB}
		Let $\Omega \subset \mathbb{R}^{n}$ be  open, bounded with $\partial \Omega \in C^{2}.$   Let $A \in L^{\infty}\cap \mathrm{V}\mathscr{L}_{\left( 1+ \left\lvert \log r \right\rvert \right)^{-1}}^{2,n}\left( \Omega; \mathcal{L}\left( \varLambda^{k+1}; \varLambda^{k+1}\right)\right),$ $B \in L^{\infty}\cap \mathrm{V}\mathscr{L}_{\left( 1+ \left\lvert \log r \right\rvert \right)^{-1}}^{2,n}\left( \Omega; \mathcal{L}\left( \varLambda^{k}; \varLambda^{k}\right)\right)$ be uniformly Legendre-elliptic with ellipticity constants $\gamma_{A}, \gamma_{B} >0$, respectively. Let $F \in \mathrm{BMO}\left( \Omega; \varLambda^{k+1}\right)$ and $G \in \mathrm{BMO}\left( \Omega; \varLambda^{k}\right).$ Let $u_{0}\in L^{2}\left( \Omega ; \varLambda^{k}\right)$ be such that $du_{0} \in L^{2}\left( \Omega ; \varLambda^{k+1}\right)$, $u_{0} \in \mathrm{BMO}\left( \Omega; \varLambda^{k}\right)$ and $du_{0} \in \mathrm{BMO}\left( \Omega; \varLambda^{k+1}\right).$
		Then there exists $u \in W^{1,2}\left( \Omega ; \varLambda^{k}\right),$ which is a weak solution for the following 
		\begin{align}\label{Main AB}
			\left\lbrace \begin{aligned}
				d^{\ast} \left( A\left( x \right) du \right)  &= d^{\ast} F &&\text{ in } \Omega, \\
				d^{\ast} \left( B\left(x\right) u\right) &= d^{\ast}G &&\text{  in } \Omega, \\
				\nu\wedge u &= \nu \wedge u_{0} &&\text{ on } \partial\Omega,
			\end{aligned}	\right. 
		\end{align}
		such that $u \in \mathrm{BMO} \left( \Omega; \varLambda^{k}\right)$ and $d u \in \mathrm{BMO} \left( \Omega,; \varLambda^{k + 1}\right)$ and there exists a constant $C>0,$ depending only on 
		\begin{align*}
			n, k, N, \gamma_{A},\gamma_{B}, \Omega, \Theta^{A},\Theta^{B}, \left\lVert A\right\rVert_{L^{\infty}},\left\lVert B\right\rVert_{L^{\infty}}, \left[ A \right]_{\mathscr{L}^{2,n}_{\frac{1}{1+ \left\lvert \log r \right\rvert}}}, \left[ B \right]_{\mathscr{L}^{2,n}_{\frac{1}{1+ \left\lvert \log r \right\rvert}}}, 
		\end{align*} such that we have the estimates
		\begin{multline}\label{du bmo estimate A B}
			\left[ u\right]_{\mathrm{BMO}\left(\Omega\right)} + \left[ du\right]_{\mathrm{BMO}\left(\Omega\right)}  \\\leq   C \left(  \left\lVert u \right\rVert_{L^{2}\left(\Omega\right)} + \left[ \left( F, G, u_{0}, du_{0} \right) \right]_{\mathrm{BMO}\left( \Omega\right)} + \left\lVert \left( F, G, u_{0}, du_{0} \right) \right\rVert_{L^{2}\left(\Omega\right)}\right) .
		\end{multline}
		Moreover, the solution is unique if and only if $\mathcal{H}^{k}_{T}\left(  \Omega;\varLambda^{k}\right) = \left\lbrace 0 \right\rbrace.$ 
	\end{theorem}
	\begin{remark}
		If $\mathcal{H}^{k}_{T} \neq \left\lbrace 0 \right\rbrace,$ then the set of weak solution $v \in W^{1,2}_{d^{\ast}, T}\left( \Omega ; \varLambda^{k}\right)$ of \eqref{Main AB} with zero data is in one to one correspondence with the set of harmonic fields $h \in \mathcal{H}^{k}_{T}$ ( see Lemma \ref{uniqueness lemma} ) and all weak solutions satisfy the estimate \eqref{du bmo estimate A B}. Also, if $\mathcal{H}^{k}_{T} = \left\lbrace 0 \right\rbrace,$ then the $L^{2}$ norm of $u$ can be dropped from the RHS of the estimate \eqref{du bmo estimate A B} in view of uniqueness.  
	\end{remark}
	\begin{proof}
		First we find $\alpha \in W^{1,2}_{d^{\ast}, T}\left( \Omega ; \varLambda^{k}\right) \cap \left( \mathcal{H}^{k}_{T}\right)^{\perp}$, a weak solution to
		\begin{align*}
			\left\lbrace \begin{aligned}
				d^{\ast} \left( A\left( x \right) d\alpha \right)  &= d^{\ast} F - d^{\ast}\left( A \left(x\right) du_{0}\right)&&\text{ in } \Omega, \\
				d^{\ast} \alpha &= 0 &&\text{  in } \Omega, \\
				\nu\wedge \alpha &= 0 &&\text{ on } \partial\Omega.
			\end{aligned}	\right. 
		\end{align*}
		Exactly as before, Theorem \ref{Main theorem BMO linear A only} implies $\nabla \alpha \in \mathrm{BMO}$ with corresponding estimates. 
		Now we find $\beta \in W^{1,2}_{d^{\ast}, T}\left( \Omega ; \varLambda^{k}\right)\cap \left( \mathcal{H}^{k}_{T}\right)^{\perp},$ a weak solution to
		\begin{align*}
			\left\lbrace \begin{aligned}
				d^{\ast} \left( B\left( x \right) d\beta \right)  &= d^{\ast} G - d^{\ast}\left( B \left(x\right) \alpha\right) - d^{\ast}\left( B \left(x\right) u_{0} \right) &&\text{ in } \Omega, \\
				d^{\ast} \beta &= 0 &&\text{  in } \Omega, \\
				\nu\wedge \beta &= 0 &&\text{ on } \partial\Omega.
			\end{aligned}	\right. 
		\end{align*} 
		Again, Theorem \ref{Main theorem BMO linear A only} implies $\nabla \beta \in \mathrm{BMO}$ with corresponding estimates. Since $\nu\wedge \beta = 0 $ on $\partial\Omega$ implies $\nu\wedge d\beta = 0 $ 
		on $\partial\Omega,$ it is easy to verify $u = \alpha + d\beta + u_{0}$ solves \eqref{Main AB}, $u \in \mathrm{BMO}$ and $du \in \mathrm{BMO}$ along with the estimates. 
	\end{proof}
	As a consequence, we have the following regularity result for the general Hodge-maxwell system. 
	\begin{theorem}\label{Main theorem Maxwell BMO linear AB}
		Let $\Omega \subset \mathbb{R}^{n}$ be  open, bounded with $\partial \Omega \in C^{2}.$   Let $A \in L^{\infty}\cap \mathrm{V}\mathscr{L}_{\left( 1+ \left\lvert \log r \right\rvert \right)^{-1}}^{2,n}\left( \Omega; \mathcal{L}\left( \varLambda^{k+1}; \varLambda^{k+1}\right)\right),$ $B \in L^{\infty}\cap \mathrm{V}\mathscr{L}_{\left( 1+ \left\lvert \log r \right\rvert \right)^{-1}}^{2,n}\left( \Omega; \mathcal{L}\left( \varLambda^{k}; \varLambda^{k}\right)\right)$ be uniformly Legendre-elliptic with ellipticity constants $\gamma_{A}, \gamma_{B} >0$, respectively. Let $\lambda \geq 0,$ $F \in \mathrm{BMO}\left( \Omega; \varLambda^{k+1}\right)$ and $G \in \mathrm{BMO}\left( \Omega; \varLambda^{k}\right).$ Let $u_{0}\in L^{2}\left( \Omega ; \varLambda^{k}\right)$ be such that $du_{0} \in L^{2}\left( \Omega ; \varLambda^{k+1}\right)$, $u_{0} \in \mathrm{BMO}\left( \Omega; \varLambda^{k}\right)$ and $du_{0} \in \mathrm{BMO}\left( \Omega; \varLambda^{k+1}\right).$
		If there exists $u \in L^{2}\left( \Omega ; \varLambda^{k}\right)$ with $du \in  L^{2}\left( \Omega ; \varLambda^{k+1}\right)$ such that $u$ is a weak solution for the following 
		\begin{align}\label{Main Maxwell AB}
			\left\lbrace \begin{aligned}
				d^{\ast} \left( A\left( x \right) du \right)  &= \lambda B\left(x\right) u  - \lambda G + d^{\ast} F &&\text{ in } \Omega, \\
				d^{\ast} \left( B\left(x\right) u\right) &= d^{\ast}G &&\text{  in } \Omega, \\
				\nu\wedge u &= \nu \wedge u_{0} &&\text{ on } \partial\Omega,
			\end{aligned}	\right. 
		\end{align}
		then $u \in \mathrm{BMO} \left( \Omega; \varLambda^{k}\right)$ and $d u \in \mathrm{BMO} \left( \Omega,; \varLambda^{k + 1}\right)$ and there exists a constant $C>0,$ depending only on 
		\begin{align*}
			n, k, N, \lambda, \gamma_{A},\gamma_{B}, \Omega, \Theta^{A},\Theta^{B}, \left\lVert A\right\rVert_{L^{\infty}},\left\lVert B\right\rVert_{L^{\infty}}, \left[ A \right]_{\mathscr{L}^{2,n}_{\frac{1}{1+ \left\lvert \log r \right\rvert}}}, \left[ B \right]_{\mathscr{L}^{2,n}_{\frac{1}{1+ \left\lvert \log r \right\rvert}}}, 
		\end{align*} such that we have the estimates
		\begin{multline}\label{du bmo estimate Maxwell A B}
			\left[ u\right]_{\mathrm{BMO}\left(\Omega\right)} + \left[ du\right]_{\mathrm{BMO}\left(\Omega\right)}  \\\leq   C \left(  \left\lVert u \right\rVert_{L^{2}\left(\Omega\right)} + \left[ \left( F, G, u_{0}, du_{0} \right) \right]_{\mathrm{BMO}\left( \Omega\right)} + \left\lVert \left( F, G, u_{0}, du_{0} \right) \right\rVert_{L^{2}\left(\Omega\right)}\right) .
		\end{multline}
	\end{theorem}
	\begin{proof}
		Since $u, u_{0} \in L^{2}\left( \Omega ; \varLambda^{k}\right),$ by standard Hodge decomposition theorem, we can write 
		\begin{align*}
			u -u_{0} = d\alpha + d^{\ast}\beta + h ,
		\end{align*}
		where $\alpha \in  W_{T}^{1,2}\left(  \Omega;\varLambda^{k-1}\right),  $ $\beta \in  W_{T}^{1,2}\left(  \Omega;\varLambda^{k+1}\right),  $ $h \in \mathcal{H}^{k}_{T}\left(  \Omega;\varLambda^{k}\right)$ and  
		\begin{align*}
			d^{\ast}\alpha = 0  \quad \text{ and }\quad  d\beta = 0 \qquad \text{ in  } \Omega. 
		\end{align*}
		Now since $\nu \wedge u - u_{0} = 0$ on $\partial\Omega,$ we have $\nu \wedge d^{\ast}\beta = 0$ on $\partial \Omega.$
		As $du, du_{0} \in L^{2}$ and $\beta \in W_{T}^{1,2}\left(  \Omega;\varLambda^{k+1}\right)$ is a weak solution to 
		\begin{align*}
			\left\lbrace \begin{aligned}
				dd^{\ast}\beta &= du - du_{0} &&\text{ in } \Omega, \\
				d\beta &= 0 &&\text{ in } \Omega, \\
				\nu \wedge \beta &=0 &&\text{ on } \partial\Omega, \\
				\nu \wedge d^{\ast}\beta &=0 &&\text{ on } \partial\Omega, 
			\end{aligned}\right. 
		\end{align*}
		we deduce $\beta \in  W^{2,2}\left(  \Omega;\varLambda^{k+1}\right)$ ( see Theorem 10 in \cite{Sil_linearregularity}, our system is the Hodge dual ). Now, we see that $\alpha \in  W_{T}^{1,2}\left(  \Omega;\varLambda^{k-1}\right)$ satisfies 
		\begin{align*}
			\left\lbrace \begin{aligned}
				d^{\ast} \left( B\left(x\right) d\alpha\right) &= d^{\ast}G - 	d^{\ast} \left( B\left(x\right) \left[ u_{0} +h + d^{\ast}\beta \right]\right)   &&\text{ in } \Omega, \\
				d^{\ast}\alpha &= 0 &&\text{ in } \Omega, \\
				\nu \wedge \alpha &=0 &&\text{ on } \partial\Omega. \\
			\end{aligned}\right. 
		\end{align*}
		As $d^{\ast}\beta \in W^{1,2} \hookrightarrow L^{\frac{2n}{n-2}},$ Theorem \ref{Main theorem Lp linear A u0} implies $\alpha \in W^{1, \frac{2n}{n-2}}.$ This implies $u \in L^{\frac{2n}{n-2}}.$ But $\psi = d^{\ast}\beta$ is a weak solution to 
		\begin{align*}
			\left\lbrace \begin{aligned}
				d^{\ast} \left( A\left( x \right) d\psi \right)  &= \lambda B\left(x\right) u  - \lambda G + d^{\ast} F &&\text{ in } \Omega, \\
				d^{\ast} \psi &= d^{\ast}G &&\text{  in } \Omega, \\
				\nu\wedge \psi &= 0 &&\text{ on } \partial\Omega.
			\end{aligned}	\right. 
		\end{align*}
		Thus, $\psi \in W^{1, \frac{2n}{n-2}}.$ This implies $du \in L^{\frac{2n}{n-2}}.$ Repeating the arguments finitely many times, we deduce that $u \in L^{q},$ where $q<n$ is such that $\frac{nq}{n-q} >n.$ Hence, by Theorem 14 in \cite{Sil_linearregularity}, we can find $\phi \in W^{1,q}$ such that 
		\begin{align*}
			\left\lbrace \begin{aligned}
				d^{\ast}\phi &= \lambda B\left(x\right)u - \lambda G    &&\text{ in } \Omega, \\
				d\phi &= 0 &&\text{ in } \Omega, \\
				\nu \wedge \phi &=0 &&\text{ on } \partial\Omega. \\
			\end{aligned}\right. 
		\end{align*}
		Hence, $\phi \in \mathrm{BMO}$ and from \eqref{Main Maxwell AB}, we have 
		\begin{align*}
			\left\lbrace \begin{aligned}
				d^{\ast} \left( A\left( x \right) du \right)  &= d^{\ast}\phi + d^{\ast} F &&\text{ in } \Omega, \\
				d^{\ast} \left( B\left(x\right) u\right) &= d^{\ast}G &&\text{  in } \Omega, \\
				\nu\wedge u &= \nu \wedge u_{0} &&\text{ on } \partial\Omega.
			\end{aligned}	\right. 
		\end{align*}
		Now applying Theorem \ref{Main theorem BMO linear AB} completes the proof. 
	\end{proof}
	
	\begin{proof}[\textbf{Proof of Theorem \ref{BMO for Maxwell in 3dim}}]
		By eliminating $H = \frac{i}{\omega}\left[ \mu^{-1}\operatorname{curl}E - \mu^{-1}J_{m} \right],$ we can rewrite the Maxwell system  as the second order system 
		\begin{align*}
			\left\lbrace \begin{aligned}
				\operatorname*{curl} ( \mu^{-1} \operatorname*{curl} E  ) &=  \omega^2 \varepsilon E -i\omega J_{e} + \operatorname*{curl}\left( \mu^{-1} J_{m}\right)    
				&&\text{ in } \Omega, \\
				\operatorname*{div} ( \varepsilon E ) &= \frac{i}{\omega}\operatorname*{div} J_{e} &&\text{ in } \Omega, \\
				\nu \times E &= \nu \times E_{0} &&\text{  on } \partial\Omega,
			\end{aligned} 		\right. 
		\end{align*}
		By Remark \ref{Remark about Mean Oscillation space properties}(iii), $\mu^{-1}J_{m} \in \mathrm{BMO}.$ 	Using Lemma \ref{inverse also elliptic and same space}, the result follows from Theorem \ref{Main theorem Maxwell BMO linear AB} by taking $A = \mu^{-1}$ and $B= \varepsilon$.
	\end{proof}

\end{document}